\documentclass[a4paper, 11pt]{article}

\usepackage
{
makeidx,epsfig, amstext, setspace, amsmath, amssymb,amsthm,amsfonts,times,latexsym,mathptmx,algorithm,algorithmic,
wrapfig,bbm,graphicx,color,enumerate,endnotes, tikz
}
\usepackage{caption}
\usepackage{subcaption}
\usepackage{tabularx}
\usepackage{hyperref}
\usepackage[noabbrev,capitalise]{cleveref}
\usepackage{thmtools}
\usepackage{thm-restate}

\usepackage{pgfplots}
\usetikzlibrary{positioning}
\usetikzlibrary{fillbetween}
\tikzset{every picture/.style={line width=0.8pt}}

\usepackage{enumitem}

\usepackage{xcolor}
\definecolor{green}{rgb}{0,0.8,0}

\definecolor{cMaroon}{HTML}{93152a}

\newcommand{\defn}[1]{\textcolor{cMaroon}{\emph{#1}}}

\usepackage[top=2.5cm, bottom= 2.5cm, left= 2.5cm, right= 2.5cm]{geometry}
\newtheorem{theorem}{Theorem}

\newtheorem{LE}[theorem]{Lemma}

\newtheorem{CN}[theorem]{Conjecture}

\newtheorem*{DEF}{Definition}

\newcounter{claim_nb}[theorem]
\newtheorem{claim}[claim_nb]{Claim}
\newtheorem*{claim*}{Claim}

\DeclareMathOperator{\dist}{dist}

\newcommand{\ignore}[1]{}

 \newcommand{\footremember}[2]{%
    \footnote{#2}
    \newcounter{#1}
    \setcounter{#1}{\value{footnote}}%
}
\newcommand{\footrecall}[1]{%
    \footnotemark[\value{#1}]%
}

\date{}

\begin{document}
\title{A Menger-type theorem for two induced paths}
\author{Sandra Albrechtsen \footremember{hamburg}{Universität Hamburg (Germany), \textit{$\{$sandra.albrechtsen, raphael.jacobs, paul.knappe$\}$@uni-hamburg.de}} \and Tony Huynh \footremember{sapienza} {Sapienza Università di Roma (Italy), \textit{$\{$huynh, wollan$\}$@di.uniroma1.it}} \and Raphael W. Jacobs \footrecall{hamburg} \and Paul Knappe \footrecall{hamburg} \and Paul Wollan \footrecall{sapienza}}

\maketitle

\begin{abstract}
We give an approximate Menger-type theorem for when a graph $G$ contains two $X-Y$ paths $P_1$ and $P_2$ such that $P_1 \cup P_2$ is an induced subgraph of $G$.  More generally, we prove that there exists a function $f(d) \in O(d)$, such that for every graph $G$ and $X,Y \subseteq V(G)$, either there exist two $X-Y$ paths $P_1$ and $P_2$ such that $\dist_G(P_1,P_2) \ge d$, or there exists $v \in V(G)$ such that the ball of radius $f(d)$ centered at $v$ intersects every $X-Y$ path.
\end{abstract}
%
\section{Introduction}

All graphs in this paper are finite and simple.

Let $G$ be a graph, and $X, Y \subseteq V(G)$.  An \defn{$X-Y$ path} is a path with one end in $X$, one end in $Y$, and no internal vertex in $X \cup Y$.  Note that by the definition, a single vertex of $X \cap Y$ is an $X-Y$ path.  If either~$X$ or $Y$ is empty, we will use the convention that no $X-Y$ path exists.  If $v$ is a single vertex in $G$, we will refer to a $v-X$ path as shorthand for a $\{v\} - X$ path.  Let $r \ge 0$ be an integer, and let $u,v \in V(G)$. The \defn{distance between $u$ and $v$}, denoted $\dist_G(u,v)$, is the length of a shortest $u-v$ path in $G$.  The \defn{ball of radius $r$ around $v$} is $B_G(v,r) := \{u \in V(G): \dist_G(u,v) \le r\}$.  More generally, if $H$ is a subgraph of $G$, we define $B_G(H,r)=\bigcup_{v \in V(H)} B_G(v, r)$.  For two subgraphs $H_1, H_2$ of $G$, the \defn{distance between $H_1$ and $H_2$}, denoted $\dist_G(H_1, H_2)$, is the minimum length of a $V(H_1)-V(H_2)$ path.

Given a graph $G$ and $X, Y \subseteq V(G)$, Menger's classic result says that the maximum number of pairwise vertex disjoint $X-Y$ paths is equal to the minimum size of a set $Z$ of vertices intersecting every $X-Y$ path~\cite{Menger1927}.  We consider the problem of finding many distinct paths between $X$ and $Y$ requiring not just that the paths be pairwise disjoint, but that the paths be pairwise far apart in $G$.  Let $H_1$ and $H_2$ be subgraphs of a graph $G$.  We say $H_1$ and $H_2$ are \defn{anti-complete} if there does not exist an edge of $G$ with one end in~$V(H_1)$ and one end in $V(H_2)$.  No analog of Menger's theorem is known for pairwise anti-complete paths linking two sets of vertices.  Such an exact characterization of when such paths exist is unlikely given that it is NP-complete to decide whether there exist two disjoint anti-complete $X-Y$ paths \cite{bienstock91,bienstock92}.  Note that a graph $G$ contains two disjoint anti-complete $X-Y$ paths if an only if it has two disjoint $X-Y$ paths which are an induced subgraph (take two disjoint anti-complete $X-Y$ paths $P_1$ and $P_2$ with~$|V(P_1)| \cup |V(P_2)|$ minimal).

We will show the following weak characterization of the induced two paths problem.

\begin{theorem}\label{thm:main}
There exists a constant $c$ such that for all graphs $G$, and all $X, Y \subseteq V(G)$, either there exist two disjoint $X-Y$ paths $P_1$, $P_2$ such that $P_1 \cup P_2$ is an induced subgraph of $G$, or there exists $z \in V(G)$ such that $B_G(z, c)$ intersects every $X - Y$ path.
\end{theorem}

Theorem \ref{thm:main} immediately follows from a more general result where we ask that the two paths be at distance at least $d$.
\begin{restatable}{theorem}{general}
\label{thm:main_general}
Let $c=129$.  For all graphs $G$, $X, Y \subseteq V(G)$, and integers $d \ge 1$, either there exist two disjoint $X-Y$ paths $P_1$, $P_2$ such that $\dist_G(P_1, P_2) \ge d$ or there exists $z \in V(G)$ such that $B_G(z, cd)$ intersects every $X - Y$ path.
\end{restatable}

After submitting our paper, we learned that Georgakopoulos and Papasoglu have independently proved~\cref{thm:main_general} (with a slightly worse constant of $272$)~\cite{GP23}.

We conjecture that Theorem \ref{thm:main_general} should generalize to an arbitrary number of $X-Y$ paths at pairwise distance $d$.

\begin{CN} \label{kinducedpaths}
There exists a constant $c$ satisfying the following. For all graphs $G$, $X, Y \subseteq V(G)$, and integers $d, k \ge 1$, either there exist $k$ disjoint $X-Y$ paths $P_1, \dots, P_k$ such that $\dist_G(P_i, P_j) \ge d$ for all distinct $i,j$ or there exists a set $Z \subseteq V(G)$ of size at most $k-1$ such that $B_G(Z, cd)$ intersects every $X - Y$~path.
\end{CN}

Rose McCarty and Paul Seymour proved that the $d=3$ case of~\cref{kinducedpaths} actually implies the general case. The proof has not been published, and we thank them for allowing us to include their proof in this paper.

\begin{theorem}[McCarty and Seymour, 2023] \label{d=3reduction}
If~\cref{kinducedpaths} holds for $d=3$ with constant $c$, then it holds for all $d \geq 3$ with constant $3c$.
\end{theorem}

\begin{proof}
Let $H$ be the $d$th power of $G$, that is, $V(H)=V(G)$, and $u$ and $v$ are adjacent in $H$ if and only if $\dist_G(u,v) \le d$. If there exist $k-1$ balls of radius $3c$ in $H$ whose union intersects all $X-Y$ paths in $H$, then there are $k-1$ balls of radius $3cd$ in $G$ so that their union intersects all $X-Y$ paths in $G$. Otherwise (by assumption), there are $X-Y$ paths $P_1, \dots, P_k$ in $H$ such that $\dist_H(P_i,P_j) \ge 3$ for all distinct $i,j$. Suppose $P_1=v_0 \dots v_t$. For each $0 \le i \le t-1$, there is a path $S_i$ of length at most $d$ in $G$ between $v_i$ and $v_{i+1}$. Thus, there is a $v_0-v_t$ path $P_1^*$ in $G$ such that $P_1^* \subseteq \bigcup_{i=0}^{t-1} S_i$.  We define $P_2^*, \dots P_k^*$ similarly.
We claim that $\dist_G(P_i^*, P_j^*) \ge d$ for all distinct $i,j$.  Towards a contradiction, suppose $\dist_G(P_i^*, P_j^*) \le d$ for some $i<j$.  Then there is a $V(P_i)-V(P_j)$ path in $G$ of length at most $\frac{d}{2}+d+\frac{d}{2} \le 2d$. Hence, $\dist_H(P_i,P_j) \le 2$, which is a contradiction.
\end{proof}

We would like to point out that Nguyen, Scott, and Seymour~\cite{NSS24} have recently found a counterexample to~\cref{kinducedpaths}.  

Our proof shows that we may take $c=129$ in~\cref{thm:main_general}.  On the other hand, the following lemma shows that $c$ must be at least $3/2$.

\begin{LE}\label{lem:lowerbound}
For every integer $d \geq 2$, there exists a graph $G$ and $X, Y \subseteq V(G)$ such that there do not exist two disjoint $X-Y$ paths $P_1, P_2$ such that $dist_G(P_1, P_2) \ge d$, nor does there exist a vertex $z \in V(G)$ such that $B_G(z, 3d/2-2)$ intersects every $X-Y$ path.
\end{LE}

\begin{proof}
Consider the graph $H$ in \cref{fig1}.  If we let $X = \{x_1, x_2, x_3\}$ and $Y = \{y_1, y_2, y_3\}$, then there do not exist two disjoint $X-Y$ paths $P_1$ and $P_2$ such that $\dist_{H}(P_1, P_2) \ge 2$.  However, neither does $H$ have a vertex $z$ such that $B_{H}(z, 1)$ intersects every $X-Y$ path.
\begin{figure}[ht]
\centering
\begin{tikzpicture}[every node/.style={draw, circle, fill=black, inner sep=1.5pt}]

  \node [label=left:$x_1$](11)  at (1,-1) {};
  \node[label=left:$x_2$](12)  at (1,-2) {};
  \node[label=left:$x_3$](13) at (1,-3) {};
  \foreach \x in {2,...,6} {
    \foreach \y in {1,...,3} {
      \node (\x\y) at (\x,-\y) {};
    }
    }
  \node [label=right:$y_1$](71) at (7,-1) {};
  \node[label=right:$y_2$](72)  at (7,-2) {};
  \node[label=right:$y_3$](73) at (7,-3) {};
  \foreach \x in {2,...,6} {
    \foreach \y in {1,...,3} {
      \node (\x\y) at (\x,-\y) {};
    }
    }
    \foreach \x in {1,7} {
    \foreach \y in {1,2} {
      \node at (\x,-\y-0.5) {};
    }
    }

  \foreach \x in {1,...,6} {
    \foreach \y in {1,...,3} {
      \pgfmathtruncatemacro{\nextx}{\x+1}
      \draw (\x\y) -- (\nextx\y);
    }
  }

  \foreach \x in {1,...,7} {
    \foreach \y in {1,...,2} {
      \pgfmathtruncatemacro{\nexty}{\y+1}
      \draw (\x\y) -- (\x\nexty);
    }
  }

  \foreach \i in {1,...,5} {
       \pgfmathtruncatemacro{\next}{\i+2}
      \draw[bend left=20](\i1) to (\next3);
}
\end{tikzpicture}
\caption{A graph without two disjoint anti-complete $\{x_1, x_2, x_3\}-\{y_1, y_2, y_3\}$ paths and no ball of radius one hitting all $X-Y$ paths.}\label{fig1}
\end{figure}
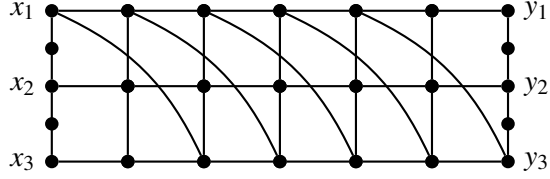
Let $G$ be the graph obtained from $H$ by subdividing each edge exactly $d-2$ times.  We claim that $G$ satisfies the lemma.  To see this, let $z$ be any vertex in $G$.
Then there is a vertex $y \in H$ of distance at most $(d-1)/2$ to $z$ in $G$.
Every other vertex of $H$ has distance at least $(d-1)/2$ to $z$ in $G$.
Since any two distinct vertices of $H$ are at distance at least $d-1$ in $G$ and $B_{H}(y,1)$ does not intersect every $X$--$Y$ path in $H$, it follows that $B_{G}(z, d-3/2+(d-1)/2)$  does not intersect every $X$--$Y$ path in $G$.
\end{proof}

We conclude this section with some necessary notation.

Let $G$ be a graph and $H_1, H_2$ two subgraphs of $G$.  The subgraph \defn{$H_1 \cup H_2$} is the subgraph with vertex set $V(H_1) \cup V(H_2)$ and edge set $E(H_1) \cup E(H_2)$.   An \defn{$H_1$-path} is a path with both ends in $H_1$ and no internal vertex or edge in $H_1$.  Let $X\subseteq V(G)$.  An \defn{$X$-path} is any $J$-path where $J$ is the subgraph of $G$ with vertex set $X$ and no edges. For a path $P = v_0v_1 \dots v_n$ and $0 \leq i \leq j \leq n$, we write $v_iP := v_iv_{i+1} \dots v_n$, $Pv_i := v_0v_1 \dots v_i$ and $v_iPv_j := v_iv_{i+1} \dots v_j$ for the corresponding subpaths of $P$.


\section{Interlaced systems of intervals}

In this section we prove a weak characterization of when a sequence of intervals contains an ``interlaced'' subsequence.  We begin with the definition of interlaced. See \cref{fig:ABrsystem} for a visualisation.

\begin{figure}[ht]
    \centering
    \begin{tikzpicture}
        \def\intervalA[#1](#2:#3)(#4:#5) 
            {
                \ifthenelse{#1=1}
                {
                    \draw[teal] (#2,1.5*#1) node[anchor=south]{#4} -- (#3, 1.5*#1) node[anchor=south]{#5};
                    \draw[teal,dotted] (#2,0) -- (#2,1.5*#1-#1*0.15);
                    \draw[teal,dotted] (#3, 1.5*#1-#1*0.15) -- (#3, 0);
                    \draw[teal] (#2+0.075,1.5*#1-#1*0.1) -- (#2,1.5*#1-#1*0.1) -- (#2, 1.5*#1+#1*0.1) -- (#2+0.075,1.5*#1+#1*0.1);
                    \draw[teal] (#3-0.075,1.5*#1-#1*0.1) -- (#3,1.5*#1-#1*0.1) -- (#3, 1.5*#1+#1*0.1) -- (#3-0.075,1.5*#1+#1*0.1);
                }
                {
                    \draw[teal] (#2,1.5*#1) node[anchor=north]{#4} -- (#3, 1.5*#1) node[anchor=north]{#5};
                    \draw[teal,dotted] (#2,0) -- (#2,1.5*#1-#1*0.15);
                    \draw[teal,dotted] (#3, 1.5*#1-#1*0.15) -- (#3, 0);
                    \draw[teal] (#2+0.075,1.5*#1-#1*0.1) -- (#2,1.5*#1-#1*0.1) -- (#2, 1.5*#1+#1*0.1) -- (#2+0.075,1.5*#1+#1*0.1);
                    \draw[teal] (#3-0.075,1.5*#1-#1*0.1) -- (#3,1.5*#1-#1*0.1) -- (#3, 1.5*#1+#1*0.1) -- (#3-0.075,1.5*#1+#1*0.1);
                };
            }
            
        \intervalA[1](-8:-5)($a_1$:$b_1$)
        \intervalA[-1](-6:-3)($a_2$:$b_2$)
        \intervalA[1](-4:-1)($a_3$:$b_3$)
        \intervalA[-1](-2:1)($a_4$:$b_4$)

        \def\distanceA[#1](#2:#3)(#4) 
            {
                    \draw[red] (#2,#1*#4) -- (#3, #1*#4) node[midway,fill=white]{$\ge \ell$};
                    \draw[red] (#2,#1*#4-0.15) -- (#2,#1*#4+0.15);
                    \draw[red] (#3,#1*#4-0.15) -- (#3,#1*#4+0.15);
                    \draw[red,dotted] (#2,#1*#4-#1*0.15) -- (#2, 0);
                    \draw[red,dotted] (#3,#1*#4-#1*0.15) -- (#3, 0);
            }

        \distanceA[-1](-6:-4)(0.5)
        \distanceA[-1](-5:-3)(1)
        \distanceA[1](-4:-2)(0.5)
        \distanceA[1](-3:-1)(1)

        \draw[thick] (-8,0) -- (1,0);
        
        \draw[thick] (-8,0.15) -- (-8,-0.15) node[anchor=north east] {$A$};
        \draw[thick] (1,0.15) -- (1,-0.15) node[anchor=north west] {$B$};
    
        \draw[blue] (-7,0.15) -- (-7,-0.15) node[anchor=north] {$A+r$};
        \draw[blue] (0,0.15) -- (0,-0.15) node[anchor=north] {$B-r$};
        
    \end{tikzpicture}
    
    \captionof{figure}{A clean~$(A,B,r)$-system $([a_i,b_i])_{i \in [4]}$ indicated in teal which is interlaced with buffer~$\ell$.}
    \label{fig:ABrsystem}
\end{figure}
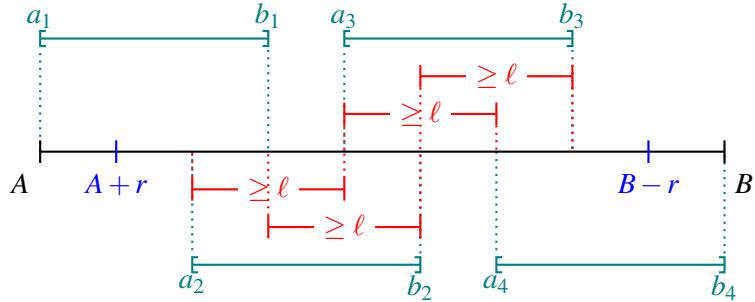

\begin{DEF}\label{Hyp1}
Let $r \in \mathbb{N}$, and $A,B \in \mathbb{Z}$ with $A \leq B$.  An \defn{$(A,B,r)$-system} is a sequence
 $([a_i, b_i])_{i \in [t]}$ of closed intervals of the real line\footnote{We will in fact only consider the integers in such intervals.} such that for all $i \in [t]$,
\begin{itemize}
\item $a_i, b_i \in \mathbb{Z}$,
\item $[a_i, b_i] \subseteq [A, B]$,
\item if $a_i \neq A$, then $a_i \geq A+r$, and
\item if $b_i \neq B$, then $b_i \leq B-r$.
\end{itemize}
Moreover, $([a_i, b_i])_{i \in [t]}$ is \defn{interlaced with buffer $\ell$} for an integer $\ell >0$ if
\begin{itemize}
\item $a_1=A, b_t=B$, and
\item $a_{i} + \ell \le a_{i+1} \le b_{i}$ and $b_{i} + \ell \leq b_{i+1}$ for all $i \in [t-1]$.
\end{itemize}
The system is \defn{clean} if $[a_i, b_i] \cap [a_j, b_j] \neq \emptyset$ implies that $|i-j| \le 1$.
\end{DEF}

Note that if $([a_i, b_i])_{i \in [t]}$ is interlaced with buffer $\ell$, then the $a_i$ and $b_i$ values are both strictly increasing.
An $(A,B,r)$-system is clean if and only if we can draw its intervals above and below the line  without any overlaps (cf. \cref{fig:notcleanABrsystem}). 

\begin{figure}[ht]
    \centering
    \begin{tikzpicture}

        \draw[thick] (-8,0) -- (1,0);
        
        \draw[thick] (-8,0.15) -- (-8,-0.15) node[anchor=north east] {$A$};
        \draw[thick] (1,0.15) -- (1,-0.15) node[anchor=north west] {$B$};
    
        \draw[blue] (-7,0.15) -- (-7,-0.15) node[anchor=north] {$A+r$};
        \draw[blue] (0,0.15) -- (0,-0.15) node[anchor=north] {$B-r$};

        \def\intervalA[#1](#2:#3)(#4:#5) 
            {
                \ifthenelse{#1=1}
                {
                    \draw[teal] (#2,#1) node[anchor=south]{#4} -- (#3, #1) node[anchor=south]{#5};
                    \draw[teal,dotted] (#2,0) -- (#2,#1-#1*0.15);
                    \draw[teal,dotted] (#3, #1-#1*0.15) -- (#3, 0);
                    \draw[teal] (#2+0.075,#1-#1*0.1) -- (#2,#1-#1*0.1) -- (#2, #1+#1*0.1) -- (#2+0.075,#1+#1*0.1);
                    \draw[teal] (#3-0.075,#1-#1*0.1) -- (#3,#1-#1*0.1) -- (#3, #1+#1*0.1) -- (#3-0.075,#1+#1*0.1);
                }
                {
                    \draw[teal] (#2,#1) node[anchor=north]{#4} -- (#3, #1) node[anchor=north]{#5};
                    \draw[teal,dotted] (#2,0) -- (#2,#1-#1*0.15);
                    \draw[teal,dotted] (#3, #1-#1*0.15) -- (#3, 0);
                    \draw[teal] (#2+0.075,#1-#1*0.1) -- (#2,#1-#1*0.1) -- (#2, #1+#1*0.1) -- (#2+0.075,#1+#1*0.1);
                    \draw[teal] (#3-0.075,#1-#1*0.1) -- (#3,#1-#1*0.1) -- (#3, #1+#1*0.1) -- (#3-0.075,#1+#1*0.1);
                };
            }

        \intervalA[1](-8:-5)(:)
        \intervalA[-1](-6:-3)(:)
        \intervalA[1](-4:-1)(:)
        \intervalA[-1](-2:1)(:)
        
        \draw[red] (-4.5,1.5) -- (-2.25, 1.5);
        \draw[red,dotted] (-4.5,0) -- (-4.5,1.5-1.5*0.15);
        \draw[red,dotted] (-2.25, 1.5-1.5*0.15) -- (-2.25, 0);
        \draw[red] (-4.5+0.075,1.5-1.5*0.1) -- (-4.5,1.5-1.5*0.1) -- (-4.5, 1.5+1.5*0.1) -- (-4.5+0.075,1.5+1.5*0.1);
        \draw[red] (-2.25-0.075,1.5-1.5*0.1) -- (-2.25,1.5-1.5*0.1) -- (-2.25, 1.5+1.5*0.1) -- (-2.25-0.075,1.5+1.5*0.1);
        
    \end{tikzpicture}
    
    \captionof{figure}{The~$(A,B,r)$-system indicated in teal is clean.  However, adding the red interval gives an $(A,B,r)$-system which is not clean.}
    \label{fig:notcleanABrsystem}
    
\end{figure}
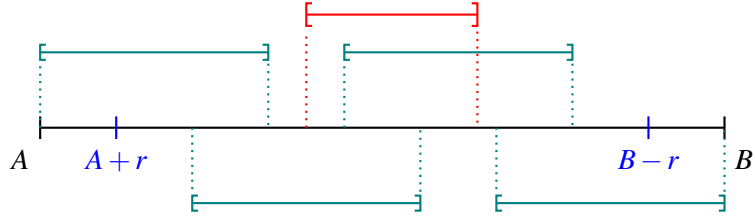

\begin{LE}\label{lem:clean}
Let $([a_i, b_i])_{i \in [t]}$ be an interlaced system with buffer $\ell$.  Then there exists a subsequence containing $[a_1, b_1]$ and $[a_t, b_t]$ which is a clean interlaced system with buffer $\ell$.
\end{LE}
\begin{proof}
We proceed by induction on $t$.  Note that the lemma trivially holds if $t = 1$.  Assume that $[a_1, b_1], \dots, [a_t, b_t]$ is not clean. Thus, there are integers $i,j \in [t]$ such that $[a_i, b_i] \cap [a_j, b_j] \neq \emptyset$ and $j \ge i+2$.  Hence, $a_j \le b_i$.  Moreover, $a_i+\ell \le a_{i+1} \le a_j$ and  $b_i+\ell \le b_{i+1} \le b_j$.
Thus, $[a_1, b_1], \dots, [a_i, b_i], [a_j, b_j], \dots, [a_t, b_t]$ is an interlaced system with buffer $\ell$ which is strictly shorter.  Applying induction to this shorter sequence yields the desired subsequence of the original sequence, proving the claim.
\end{proof}

An immediate consequence of the definition of a clean interlaced system is that every element $x \in [A,B]$ is contained in at most two intervals.

\begin{LE}\label{lem:intseq}
Let $([a_i, b_i])_{i \in [t]}$ be an $(A,B, r)$-system with $A \leq B-2r$ and let $\ell \le r$ be a positive integer.  Then either $([a_i, b_i])_{i \in [t]}$ contains an $(A,B,r)$-subsystem which is interlaced with buffer $\ell$, or there exists an integer $z \in [A+r, B-r]$ such that, for $Z = [z - 2\ell, z + 2\ell] \cap (A, B)$, no interval in $([a_i, b_i])_{i \in [t]}$ intersects both segments of $[A,B] \setminus Z$.
\end{LE}

\begin{proof}

We say an integer $y \in [A+r, B-r]$ is \defn{good} if $([a_i, b_i])_{i \in [t]}$ has an interlaced $(A,B',\ell)$-subsystem $[a_{i_1}, b_{i_1}], \dots, [a_{i_s}, b_{i_s}]$ with buffer $\ell$ such that $y \in [a_{i_s} + \ell, b_{i_s}-\ell]$.  If~$B-r$ is good, then we find the desired interlaced subsystem. Indeed, then there exists an interlaced $(A,B',\ell)$-subsystem $[a_{i_1}, b_{i_1}], \dots, [a_{i_s}, b_{i_s}]$ with buffer $\ell$ such that $B-r \in  [a_{i_s} + \ell, b_{i_s}-\ell]$.  Thus, $b_{i_s} > B - r$, which implies $b_{i_s} = B$ since $([a_i, b_i])_{i \in [t]}$ is an $(A,B,r)$-system.  Hence, we have $B'=B$ and $[a_{i_1}, b_{i_1}], \dots, [a_{i_s}, b_{i_s}]$ is as desired.

Thus, we may assume that some integer in~$[A+r, B-r]$ is not good, and we let~$z$ be the smallest such integer.  We claim that~$z$ is as desired.  So let $Z = [z - 2\ell, z + 2\ell] \cap (A, B)$ and assume for a contradiction that there exists an interval~$[a_k, b_k]$ which intersects both segments of $[A,B] \setminus Z$. We fix such an interval~$[a_k, b_k]$ for the remainder of the proof.

Suppose $a_k = A$.  Since $b_k > z + 2\ell$, the interval $[a_k, b_k]$ by itself is an interlaced system with buffer $\ell$ certifying that $z$ is good, a contradiction to our choice of $z$.  Thus, $a_k \ge A + r$. Since $a_k < z$, it follows that $a_k$ is good by the minimality of $z$.

Let $[a_{i_1}, b_{i_1}], \dots, [a_{i_s}, b_{i_s}]$ be an interlaced $(A,B',\ell)$-subsystem with buffer $\ell$ such that $a_k \in  [a_{i_s} + \ell, b_{i_s}-\ell]$.
First, we observe that $b_{i_s} < z + \ell$.  Otherwise, $z \in [a_{i_s} + \ell, b_{i_s} - \ell]$, contradicting the fact that $z$ is not good.  Since $a_k \in [a_{i_s} + \ell, b_{i_s} - \ell]$ and $b_k \ge z + 2\ell \ge b_{i_s} + \ell$, it follows that $[a_{i_1}, b_{i_1}], \dots, [a_{i_s}, b_{i_s}], [a_k, b_k]$ is an interlaced system with buffer $\ell$.  However, $z \in [a_k + \ell, b_k - \ell]$ implying that $z$ is good.  This again contradicts our choice of $z$ and thus completes the proof.
\end{proof}


\section{Proof of the Main Theorem}

We now prove our main theorem, which we restate for convenience.

\general*

\begin{proof}
Let $d_1=5d$ and $d_2=62d$.  If $G$ has no $X-Y$ path, we may take $z$ to be any vertex of $G$.  Thus, let $P=v_0 \dots v_n$ be a shortest $X-Y$ path in $G$.  Let $\mathcal{H}$ be the family of  components of $G-B_G(P, d_1)$.  If some $H \in \mathcal{H}$ contains an $X-Y$ path $W$, then we may take $P_1=P$ and $P_2=W$.  Note that $\dist_G(P_1, P_2) \ge d_1 \ge d$.  Thus, we may assume that no $H \in \mathcal{H}$ contains vertices of both $X$ and $Y$.

For each $H \in \mathcal{H}$, we define $a(H)$ to be $-d_2$ if $H$ contains a vertex of $X$.  Otherwise, $a(H)$ is the smallest index $j$ such that $\dist_G(v_j,H) = d_1+1$.  Similarly, we define $b(H)$ to be $n+d_2$ if $H$ contains a vertex of $Y$.  Otherwise, $b(H)$ is the largest index $j$ such that $\dist_G(v_j,H) = d_1+1$.  For each $H \in \mathcal{H}$, we define~$I(H)$ to be the interval $[a(H), b(H)]$ (see \cref{fig:componentsH_i}).  Let $\mathcal{H}' \subseteq \mathcal{H}$ be the components which correspond to the maximal intervals in $I(\mathcal{H}):=\{I(H): H \in \mathcal{H}\}$.    Observe that $I(\mathcal{H}'):=\{I(H): H \in \mathcal{H}'\}$ is a $(-d_2, d_2+n, d_2)$-system.

By~\cref{lem:intseq}, either there is an ordering $H_1, \dots, H_m$ of a subset of $\mathcal{H}'$ such that $I(H_1), \dots, I(H_m)$ is an interlaced system with buffer $d_2$, or there exists an integer $j \in [0, n]$ such that no interval in $I(\mathcal{H}')$
intersects both segments of $[-d_2,d_2+n] \setminus Z$ for $Z = [j - 2 d_2, j + 2 d_2] \cap (-d_2,d_2+n)$.

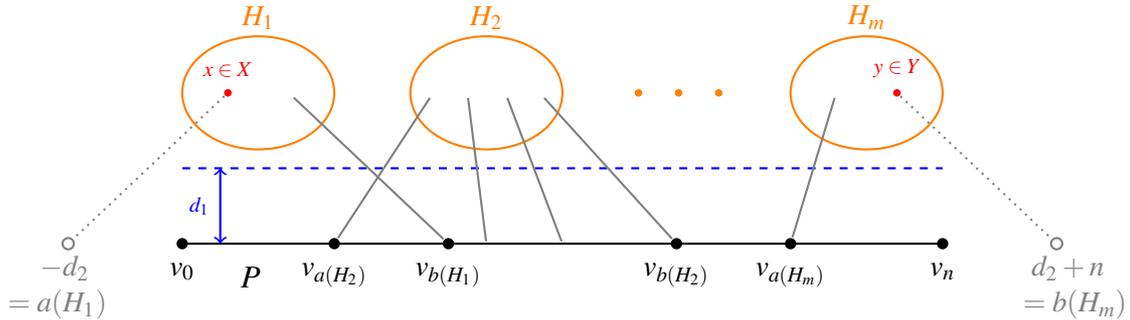
\begin{figure}[ht]
    \centering
    \begin{tikzpicture}

        \draw[thick] (-5,0.5) -- (5,0.5);

        \node (P) at (-4.1,0.5) {};
        \node[below = 0.5pt of P]{\large $P$};

        \node[circle, fill=black, inner sep=1.5pt](v0) at (-5,0.5) {};
        \node[below = 0.5pt of v0]{$v_0$};

        \node[circle, fill=black, inner sep=1.5pt](va2) at (-3,0.5) {};
        \node[below = 0.5pt of va2]{$v_{a(H_2)}$};

        \node[circle, fill=black, inner sep=1.5pt](vb1) at (-1.5,0.5) {};
        \node[below = 0.5pt of vb1]{$v_{b(H_1)}$};

        \node[circle, fill=black, inner sep=1.5pt](vb2) at (1.5,0.5) {};
        \node[below = 0.5pt of vb2]{$v_{b(H_2)}$};

        \node[circle, fill=black, inner sep=1.5pt](vam) at (3,0.5) {};
        \node[below = 0.5pt of vam]{$v_{a(H_m)}$};

        \node[circle, fill=black, inner sep=1.5pt](vn) at (5,0.5) {};
        \node[below = 0.5pt of vn]{$v_n$};


        \path[name path = d1, draw, white] (-6,1.5) -- (6,1.5); 
        \path[name path = d2, draw, blue, dashed] (-5,1.5) -- (5,1.5);
        \draw[blue, <->] (-4.5,1.5) -- (-4.5,0.5)
        node[midway,left]{\scriptsize $d_1$};


        \draw[orange] (-4,2.5) ellipse (1 and 0.75);
        \node[text=orange] at (-4, 3.5) {$H_1$};
        \node[circle, fill=red, inner sep=1pt](xX) at (-4.4,2.5) {};
        \node[red, above = 0pt of xX]{\scriptsize $x \in X$};

        \draw[orange] (-1,2.5) ellipse (1 and 0.75);
        \node[text=orange] at (-1, 3.5) {$H_2$};

        \node[circle, fill=orange, inner sep=1pt](dot) at (1,2.5) {};
        \node[circle, fill=orange, inner sep=1pt, right= 0.4 of dot](dot2) {};
        \node[circle, fill=orange, inner sep=1pt, right= 0.4 of dot2] {};

        \draw[orange] (4,2.5) ellipse (1 and 0.75);
        \node[text=orange] at (4, 3.5) {$H_m$};
        \node[circle, fill=red, inner sep=1pt](yY) at (4.4,2.5) {};
        \node[red, above = 0pt of yY]{\scriptsize $y \in Y$};


        \def\grayArrow(#1:#2) 
            {
                \path[name path = line1, draw, gray] (#1) -- (#2);
            }

        \node[circle, draw=gray, inner sep=1.5pt](va1) at (-6.5,0.5) {};
        \node[gray, below = 0.5pt of va1]{};
        \node[align=center,gray, below = 0.5pt of va1] (lab) at (va1) {$-d_2$~~\\$= a(H_1)$~~~~};
        \draw[gray, dotted] (xX) -- (va1);

        \node[inner sep=1.5pt](h1) at (-3.6, 2.5) {};
        \grayArrow(h1:vb1)

        \node[inner sep=1.5pt](h2a) at (-1.7, 2.5) {};
        \grayArrow(h2a:va2)

        \node[inner sep=1.5pt](h2b) at (-1.25, 2.5) {};
        \node[inner sep=0pt](v2b) at (-1,0.5) {};
        \grayArrow(h2b:v2b)

        \node[inner sep=1.5pt](h2c) at (-0.75, 2.5) {};
        \node[inner sep=0pt](v2c) at (0,0.5) {};
        \grayArrow(h2c:v2c)

        \node[inner sep=1.5pt](h2d) at (-0.3, 2.5) {};
        \grayArrow(h2d:vb2)

        \node[inner sep=1.5pt](hm) at (3.6, 2.5) {};
        \grayArrow(hm:vam)
        
        \node[circle, draw=gray, inner sep=1.5pt](vbm) at (6.5,0.5) {};
        \node[gray, below = 0.5pt of vbm]{};
        \node[align=center,gray, below = 0.5pt of vbm] (lab) at (vbm) {~~~$d_2+n$\\~~~~~$= b(H_m)$};
        \draw[gray, dotted] (yY) -- (vbm);

    \end{tikzpicture}
    
    \caption{The components~$H_i$ along with their respective~$a(H_i)$ and~$b(H_i)$ represented by the corresponding vertex in $P$ or a vertex in $X$ or $Y$.}
    \label{fig:componentsH_i}
\end{figure}

Suppose the latter holds. We claim that $B_G(z, cd)$ intersects all the $X-Y$ paths for $z := v_j$. 
Suppose not and let $Q$ be an $X-Y$ path which avoids $B_G(z, cd)$.
Let $R_1 := Pv_{j-2d_2-1}$ and $R_2 :=  v_{j+2d_2+1}P$.
Since $cd \geq 2d_2 + d_1$, the path $Q$ in particular avoids $B_G(v_{j-2d_2} P v_{j+2d_2}, d_1) \supseteq  B_G(P,d_1) \setminus (B_G(R_1,d_1) \cup B_G(R_2,d_1))$.

For $i \in \{1,2\}$ let $\mathcal{H}_i$ be those components $H \in \mathcal{H}$ which have a neighbour in $B_G(R_i,d_1)$.  
By assumption, no interval in $I(\mathcal{H}')$ intersects both segments of $[-d_2,d_2+n] \setminus Z$. Since  $\mathcal{H}' \subseteq \mathcal{H}$ are the components which correspond to the maximal intervals in $I(\mathcal{H})$, we conclude that no interval in $I(\mathcal{H})$ intersects both segments of $[-d_2,d_2+n] \setminus Z$.  For each $H \in \mathcal{H}$, the choice of $a(H)$ and $b(H)$ implies that if $v_h$ is a vertex of $P$ such that $\dist_G(v_h, H)=d_1+1$, then $h \in I(H)=[a(H), b(H)]$.  Therefore, we conclude that $\mathcal{H}_1$ and $\mathcal{H}_2$ are disjoint, no $H \in \mathcal{H}_1$ meets $Y$, and no $H \in \mathcal{H}_2$ meets $X$.

Since $P$ is a shortest path, $\dist_G(v_k,v_\ell) = |k - \ell|$ for all $0 \leq k, \ell \leq n$.
Hence, $\dist_G(R_1, R_2) = 4d_2+2$.
Since an edge in $G$ from $B_G(R_1, d_1)$ to $B_G(R_2, d_1)$ would witness that $\dist_G(R_1, R_2) \leq 2d_1 + 1$ but $2d_1+1 < 4d_2+2$, it follows that no such edge exists.

Altogether, the~$X-Y$ path~$Q$ avoiding~$B_G(z, cd)$ must be contained in $B_G(R_i, d_1) \cup \bigcup \mathcal{H}_i$ for some $i \in \{1,2\}$.
We consider the case $i=1$ as the case $i=2$ is analogous.
Since no $H \in \mathcal{H}_1$ meets $Y$, the end $y$ of $Q$ in $Y$ is in $B_G(R_1, d_1)$.
Thus, $v_0 \in X$ and $y \in Y$ have distance at most $j-2d_2-1 +d_1 < n$ in $G$, which contradicts that $P$ has length~$n$ and is a shortest $X-Y$ path in $G$.
 
 Therefore, we may assume there is an ordering $H_1, \dots, H_m$ of a subset of $\mathcal{H}'$ such that $I(H_1), \dots, I(H_m)$ is an interlaced system with buffer $d_2$. 
 By~\cref{lem:clean}, we may also assume that $I(H_1), \dots, I(H_m)$ is clean.  Let $S \subseteq [m-1]$ be the set of indices $i$ such that $b(H_i) - a(H_{i+1}) \le 30d$.
 For each $i \in S$, there must be some component $C_i$ of $G-B_G(P, d_1)$ such that $a(C_i) \le a(H_{i+1})-d_2$ and $b(C_i) \ge b(H_{i})+d_2$.  Otherwise, we may take $z$ to be $v_{a(H_{i+1})}$, since $cd \ge d_2 + 30d +d_1$.  Note that $C_i \notin \{H_1, \dots, H_m\}$ for all $i \in S$.  Moreover, recall that for each $j \in [m]$, $I(H_j)$ is a maximal element of $I(\mathcal{H})$ under inclusion.  Therefore, $a(H_i) < a(C_i)$ and $b(C_i) < b(H_{i+1})$ for all $i \in S$.
 Together, we have $a(H_i) < a(C_i) < a(H_{i+1})$ and $b(H_i) < b(C_i) < b(H_{i+1})$.
 Since the $a(H_i)$ and the $b(H_i)$ are strictly increasing, this also implies that $C_{i_1} \neq C_{i_2}$ for all distinct $i_1, i_2 \in S$.
 Suppose $|S|=s$.  Re-index the elements in $\mathcal{M}:=\{H_1, \dots, H_m\}$ and $\mathcal{S}:=\{C_1, \dots, C_s\}$ so that all indices in $[m+s]$ appear exactly once, $a(H_{i-1}) \le a(C_i)$ and $b(C_i) \le b(H_{i+1})$ for all $C_i$, and the re-indexing preserves the original order of $H_1, \dots, H_m$ and of $C_1, \dots C_s$.   Rename the corresponding sequence $H_1', \dots H_{m+s}'$.

 \begin{claim}\label{dist13}
     $\dist_G(v_{b(H_j')}, v_{a(H_{j+1}')}) \geq 30d$ for every $j \in [m+s-1]$.
 \end{claim}

 \begin{proof}
     If $H_j', H_{j+1}' \in \mathcal{M}$, then the claim follows directly from the choice of $\mathcal{S}$.
     Thus, we may assume that $H_j' \in \mathcal{M}$ and $H_{j+1}' \in \mathcal{S}$, since the case $H_{j}' \in \mathcal{S}$ and $H_{j+1}' \in \mathcal{M}$ is analogous.
     Then we have that $a(H_{j+1}') + d_2 \leq a(H_{j+2}') \leq b(H_j')$.
     Thus, the claim follows from~$d_2 \geq 30d$.
\end{proof}

\begin{claim}\label{dist42}
    $\dist_G(v_{a(H'_j)}, v_{a(H'_{j+2})}) \ge d_2$ for every $j \in [m+s-2]$.
\end{claim}

\begin{proof}
    If at least two of $\{H'_j, H'_{j+1}, H'_{j+2}\}$ are in $\mathcal{M}$, then the claim follows from the choice of $\mathcal{M}$. Otherwise, $H'_j \in \mathcal{S}$ and $H'_{j+2} \in \mathcal{S}$. Then the claim follows from the choice of $\mathcal{S}$.
\end{proof}

\begin{claim} \label{dist3}
	If $q \geq p+5$, then $\dist_G(H_p', H_q') \ge d_2-2d_1+2-30d$.  Moreover, if $q \geq p+7$, then $\dist_G(H_p', H_q') \ge d_2-2d_1+2$.
\end{claim}

\begin{proof}
    Let $q \geq p+5$ and let $Q$ be a shortest $H_p'-H_q'$ path.  Let $h_p$ and $h_q$ be the ends of $Q$ in $H_p'$ and $H_q'$, respectively, and let~$h_p'$ be the neighbour of~$h_p$ on~$Q$ and~$h_q'$ the neighbour of~$h_q$ on~$Q$. Since both~$H_p'$ and~$H_q'$ are components of~$G - B_G(P, d_1)$, we have~$h_p' \in B_G(v_\ell, d_1)$ for some~$\ell \le b(H_p')$ and $h_q' \in B_G(v_r, d_1)$ for some $r \ge a(H_q')$. Then
    \[
    \dist_G(v_\ell, v_r) \le \dist_G(v_\ell, h_p') + \dist_G(h_p', h_q') + \dist_G(h_q', v_r) = 2d_1 + (||Q|| - 2).
    \]

    Suppose there are at least four elements of $\mathcal{M}$ among $H_p', \dots, H_q'$.  Since $(I(H_i))_{i=1}^k$ is a clean interlaced system with buffer~$d_2$, it follows that $r-\ell \ge d_2$.  Since $P$ is a shortest $X-Y$ path, we thus obtain $\dist_G(v_\ell,v_r)=r-\ell \ge d_2$.
    Together with $\dist_G(v_\ell, v_r) \le 2d_1 + (||Q|| - 2)$, this yields $||Q|| \ge d_2-2d_1+2$. 
    
    Note that if $q \geq p+7$, then there are at least four elements of $\mathcal{M}$ among $H_p', \dots, H_q'$.
    Thus, we may assume that $q \in \{p+5, p+6\}$ and that are at most three elements of $\mathcal{M}$ among $H_p', \dots, H_q'$.  Note that this can only happen if the sequence $H_p', \dots, H_q'$ alternates between elements in $\mathcal{M}$ and $\mathcal{S}$.  By symmetry, we may assume that $H_p' \in \mathcal{S}$.  We have $\ell \leq b(H_p') \le b(H_{p+1}')$ and $a(H_{p+5}') \leq a(H_q') \leq r$.  Since $H_{p+4}' \in \mathcal{S}$, $a(H_{p+4}') \le a(H_{p+5}')-d_2$.  Since $H_{p+2}' \in \mathcal{S}$, it follows that $\dist_G(v_{b(H_{p+1}')}, v_{a(H_{p+3}')}) \le 30d$.  Therefore, $\dist_G(v_\ell, v_r) \ge d_2-30d$.   On the other hand, $\dist_G(v_\ell, v_r) \le 2d_1 + (||Q|| - 2)$, as above.  Thus, we have~$||Q|| \ge d_2-2d_1+2-30d$, as required.
\end{proof}

 Let $H'=\bigcup_{i=1}^{m+s} H_i'$. A \defn{fruit tree} is a subgraph $\mathcal{W}:=\bigcup_{i\in I} H_i' \cup W_1 \cup \dots \cup W_k$ of $G$ (see \cref{fig:fruittree}) where
 \begin{itemize}
     \item $I$ is a subset of $[m+s]$,
     \item for all $i \in [m+s]$ and $j \in [k]$, if $V(H_i') \cap V(W_j) \neq \emptyset$, then $i \in I$,
     \item for each $i \in [k]$, $W_i$ is a $\left( \bigcup_{i\in I} H_i' \cup  \bigcup_{j=1}^{i-1} W_j\right)$-path with at least one edge, and
     \item the minor $M(\mathcal{W})$ of $\mathcal{W}$ obtained by contracting $H_i'$ to a single vertex for each $i \in I$ is a tree.
 \end{itemize}

We call the paths $W_i$ the \defn{composite} paths of $\mathcal{W}$.   
Since $M(\mathcal{W})$ is a tree, the ends of $W_{i}$ are in distinct components of $\bigcup_{j\in I} H_j' \cup \bigcup_{j=1}^{i-1} W_{j}$; in particular, they are in distinct components of $\bigcup_{j=1}^{i} W_j \setminus E(W_i)$.
For each $i \in [k]$, let $n_i$ and $m_i$ be the number of composite paths of the two components of $\bigcup_{j=1}^{i} W_j \setminus E(W_i)$ containing the ends of $W_i$.
We say that $\mathcal{W}$ is \defn{$d$-small} if $W_i$ has length at most $(\max\{n_i, m_i\}+2)d$, for all $i \in [k]$.

 \begin{claim} \label{closetoH'}
     Let $\mathcal{W}$ be a $d$-small fruit tree with composite paths $W_1, \dots, W_k$ and let $C$ be a component of $\bigcup_{i=1}^k W_i$ consisting of $\ell$ composite paths. Then for every $v \in V(C)$, $\dist_{\mathcal{W}}(v, V(H')) \le \ell d$.
 \end{claim}

 \begin{proof}
     We proceed by induction on $\ell$. Suppose $\ell=1$. Then $C = W_i$ for some $i \in [k]$. Since $\mathcal{W}$ is $d$-small, $||W_i|| \le 2d$, and both ends of $W_i$ are in $V(H')$. Therefore, one of the two subpaths of $W_i$ from $v$ to $V(H')$ will have length at most $d$. 
     
     Now suppose that~$\ell \ge 2$.
     Let $W_{i_1}, \dots, W_{i_\ell}$ with increasing~$i_j \in [k]$ be all the composite paths of $\mathcal{W}$ contained in $C$.
     Since $M(\mathcal{W})$ is a tree, the ends of $W_{i_\ell}$ are in distinct components $C_1$ and $C_2$ of $\bigcup_{j=1}^\ell W_{i_j} \setminus E(W_{i_\ell})$. For each $i \in [2]$, let $w_i$ be the end of $W_{i_\ell}$ in $C_i$ and let $\ell_i$ be the number of composite paths in $C_i$.  Note that $\ell_1+\ell_2 \le \ell-1$.  By induction, we may assume that $v \in V(W_{i_\ell})$.  Moreover, $\dist_{C_1}(w_1, V(H')) \le \ell_1d$ and \mbox{$\dist_{C_2}(w_2, V(H')) \le \ell_2d$}. Therefore, we are done unless $\dist_{W_{i_\ell}}(v, w_1) > (\ell-\ell_1)d$ and $\dist_{W_{i_\ell}}(v, w_2) > (\ell-\ell_2)d$.  But this implies $||W_{i_\ell}|| > (\ell-\ell_1)d + (\ell-\ell_2)d \ge (\ell+1)d \ge (\max\{\ell_1, \ell_2\}+2)d$, which contradicts that $\mathcal{W}$ is $d$-small.
 \end{proof}

 \begin{claim} \label{sixpaths}
     Let $\mathcal{W}$ be a $d$-small fruit tree with composite paths~$W_1, \dots, W_k$. Then each component of~$\bigcup_{i=1}^{k} W_i$ consists of at most $4$ composite paths.
 \end{claim}

 \begin{proof}
Suppose that some component of~$\bigcup_{i=1}^{k} W_i$ consists of composite paths $W_{i_1}, \dots, W_{i_\ell}$ with $\ell \ge 5$ and increasing~$i_j \in [k]$. Let $\ell' \in [\ell]$ be the smallest index such that $W_{i_1} \cup \dots \cup W_{i_{\ell'}}$ has a component $C$ with at least $5$ composite paths. 

Let $c$ be the number of composite paths of $C$.
Since $M(\mathcal{W})$ is a tree, $V(H') \cap V(C)$ contains $c+1$ vertices, no two of which are in the same $H_i'$. Thus, there exist $p,q \in [m+s]$ such that $q \geq p+c$ and $C$ contains vertices $u_p \in V(H_p')$ and $u_q \in V(H_q')$. 

Let $C_1$ and $C_2$ be the two components of~$\bigcup_{j=1}^{i_{\ell'}} W_j \setminus E(W_{i_{\ell'}})$ containing the ends of $W_{i_{\ell'}}$.
Let $i \in \{1,2\}$.
We denote by $c_i$ the number of composite paths of $C_i$.
The minimality of $\ell'$ implies that $c_i \leq 4$.  
Since $\mathcal{W}$ is $d$-small, the length of $W_{i_{\ell'}}$ is at most $(4+2)d$ and we can bound the number of edges in $C_i$ from above by $\sum_{i=1}^{c_i} (i+1)d$.
Thus, we have
\[ 
    \dist_C(u_p, u_q) \le \sum_{i=1}^{c_1} (i+1)d + (4+2)d +\sum_{i=1}^{c_2} (i+1)d
\]
by joining $u_p$ to $W_{i_{\ell'}}$ in $C_1$, walking along $W_{i_{\ell'}}$ and then joining $W_{i_{\ell'}}$ to $u_q$ in $C_2$.

If $c \ge 7$, then it follows from $c_1, c_2 \leq 4$ that
\[ 
\dist_C(u_p, u_q) \le \sum_{i=1}^{4} (i+1)d + (4+2)d +\sum_{i=1}^{4} (i+1)d =34d < d_2-2d_1+2.
\]
However, this contradicts~\cref{dist3}.

If $c \in \{5,6\}$, then it follows from $c_1 + c_2 + 1 = c \leq 6$ that
\[
\dist_C(u_p, u_q) \le \sum_{i=1}^{4} (i+1)d + (4+2)d +\sum_{i=1}^{1} (i+1)d =22d < d_2-2d_1+2-30d.
\]
Again, this is a contradiction to~\cref{dist3}.
\end{proof}

For each fruit tree $\mathcal{W} =\bigcup_{i\in I} H_i' \cup W_1 \cup \dots \cup W_k$, we define $\min(\mathcal{W}):= \min~I = \min\{i \in [m+s] : H_i' \subseteq \mathcal{W}\}$ and $\max(\mathcal{W}):= \max~I = \max\{i \in [m+s] :  H_i' \subseteq \mathcal{W}\}$.

An \defn{orchard} is a sequence $\mathcal{W}^1, \dots, \mathcal{W}^t$ of $d$-small fruit trees such that
 $\min(\mathcal{W}^1)=1$, $\max(\mathcal{W}^t)=m+s$, and $\max(\mathcal{W}^i)+1=\min(\mathcal{W}^{i+1})$ for all $i \in [t-1]$. Note that each $H_i'$ is itself a $d$-small fruit tree.  Thus, $H_1', \dots, H_{s+m}'$ is an orchard.

Among all orchards, we choose $\mathcal{W}^1, \dots , \mathcal{W}^t$ such that $t$ is minimal.

For each $i \in [t]$, let $w(i)=\max (\mathcal{W}^i)$.
For each $i \in [t-1]$ let $B_i$ be a shortest $\mathcal{W}^i-\beta_i$ path, where $\beta_i:=v_{b(H_{w(i)}')}$ and let $A_i$ be a shortest $\alpha_i-\mathcal{W}^{i+1}$ path, where $\alpha_i:=v_{a(H_{w(i)+1}')}$. For each $i \in [t-1]$, let $r_i$ be the end of $B_i$ in $\mathcal{W}^i$ and $\ell_{i+1}$ be the end of $A_i$ in $\mathcal{W}^{i+1}$.  Choose $\ell_1 \in X \cap V(H_1')$ and $r_t \in Y \cap V(H_{m+s}')$ arbitrarily, and set $\alpha_t := v_n$ and $\beta_0 := v_0$. For each $i \in [t]$, let $Q_i$ be an $\ell_i-r_i$ path in $\mathcal{W}^i$. See also \cref{fig:fruittree}.

\begin{figure}[ht]
    \centering
    \begin{tikzpicture}

        \draw[thick] (-5,0) -- (7,0);

        \node (P) at (-5,0) {};
        \node[left = 1pt of P]{$P$};

        \path[name path = d2, draw, blue, dashed] (-5,1.5) -- (7,1.5);
        \draw[blue, <->] (-4.5,1.5) -- (-4.5,0)
        node[midway,left]{\scriptsize $d_1$};

        \draw[orange] (-2.5,2.5) ellipse (1 and 0.75);
        \node[text=orange] at (-4.3, 2.5) {$H'_{w(i-1)+1}$};

        \node[circle, fill=orange, inner sep=0.5pt](dot) at (-1,2.5) {};
        \node[circle, fill=orange, inner sep=0.5pt, right= 0.2 of dot](dot2) {};
        \node[circle, fill=orange, inner sep=0.5pt, right= 0.2 of dot2] {};

        \draw[orange] (1,2.5) ellipse (1 and 0.75);

        \node[circle, fill=orange, inner sep=0.5pt](dot) at (2.5,2.5) {};
        \node[circle, fill=orange, inner sep=0.5pt, right= 0.2 of dot](dot2) {};
        \node[circle, fill=orange, inner sep=0.5pt, right= 0.2 of dot2] {};

        \draw[orange] (4.5,2.5) ellipse (1 and 0.75);
        \node[text=orange] at (6, 2.5) {$H'_{w(i)}$};

        \draw[purple,thick] (-3.75,3.25) -- (-3.75,3.5) -- (5.75,3.5) node[midway,above]{$\mathcal{W}_i$} -- (5.75,3.25);

        \node[circle, fill=gray, inner sep=1.5pt](li) at (-3.2,2.25) {};
        \node[gray, above = 0pt of li]{$\ell_i$};

        \node[circle, fill=gray, inner sep=1.5pt](ai) at (-4,0) {};
        \node[gray, below = 0pt of ai]{$\alpha_{i-1}$};

        \draw[gray] (li) -- (ai) node[midway,right]{$A_{i-1}$};

        \node[circle, fill=gray, inner sep=1.5pt](ri) at (5.2,2.25) {};
        \node[gray, above = 0pt of ri]{$r_i$};

        \node[circle, fill=gray, inner sep=1.5pt](bi) at (6,0) {};
        \node[gray, below = 0pt of bi]{$\beta_i$};

        \draw[gray] (ri) -- (bi) node[midway,right]{$B_{i}$};

        \node[](h1) at (-2,2.25) {};
        \node[](h2) at (0.5,2.25) {};
        \node[](help1) at (-1.7,1.5) {};
        \node[](help2) at (0.2,1.5) {};
        \node[](h3) at (4.3,2.25) {};

        \draw[green, line width=3pt, opacity=0.3] (li.center) -- (h1.center) to[out=270, in=120] (help1.center) to[out=300, in=180] (-0.75,1) to[out=0,in=240] (help2.center) to[out=330, in=180] (2,0.75) to[out=0, in=240] (h3.center) -- (ri.center);
        \node[text=green, opacity=1] at (4,1) {$Q_i$};

        \draw[magenta, line width=1pt] (h1.center) to[out=270, in=120] (help1.center) to[out=300, in=180] (-0.75,1) to[out=0,in=240] (help2.center) to[out=60,in=270] (h2.center);
        \draw[magenta, line width=1pt] (help2.center) to[out=330, in=180] (2,0.75) to[out=0, in=240] (h3.center);

    \end{tikzpicture}
    
    \caption{The fruit tree~$\mathcal{W}_i$ of the orchard with its corresponding components (orange) and its composite paths (magenta).}
    \label{fig:fruittree}
\end{figure}
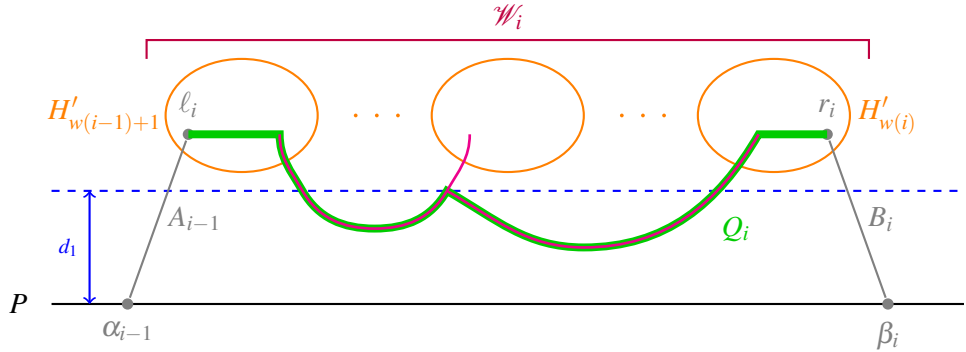

  We will now recursively define two $X-Y$ walks $P_1$ and $P_2$ such that $\dist_G(P_1, P_2) \ge d$. These walks then contain the desired paths. In order to show that our construction indeed yields $\dist_G(P_1, P_2) \ge d$, we will need to ensure that $\dist_P(P_1 \cap P, P_2 \cap P) \ge 10d$.

  Initialize $P_1:=P\alpha_1$ and $P_2:=\ell_1Q_1r_1$. Now suppose that we have already defined $P_1$ and $P_2$ such that, for some~$i \in [t]$,
  \begin{enumerate}[label=(\roman*)]
      \item\label{itm:1} $P_1$ starts in $v_0$ and ends in $\alpha_{i}$, and $P_2$ starts in $\ell_1$ and ends in $r_i$; or $P_1$ starts in $v_0$ and ends in $r_i$, and $P_2$ starts in $\ell_1 $ and ends in $\alpha_i$,
      \item\label{itm:2} $P_1 \cup P_2$ does not contain any vertices of $\bigcup_{j=w(i)+1}^{m+s} H_j'$,
      \item\label{itm:3}  among all components of $(P_1 \cup P_2) \cap P$, the component $P'$ containing $\alpha_i$ occurs last on $P$ (note that $\alpha_i$ could be the left or right end of this path),
      \item\label{itm:4} the other end of the component $P'$ is some $\beta_j$ with~$j < i$,
      \item\label{itm:5} either $\alpha_i$ occurs after $\beta_j$ along~$P$, or $\dist_P(\alpha_i, \beta_j) \leq 20d$ and if additionally $\dist_P(\alpha_i, \beta_j) > 10d$, then $\alpha_{i+1}$ occurs between $\alpha_i$ and $\beta_j$ along~$P$, and
      \item\label{itm:6} $\dist_P(P_1 \cap P, P_2 \cap P) \geq 10d$.
  \end{enumerate}
  We maintain these properties throughout the construction. 
  Note that property~\ref{itm:6} suffices in order to prove that the final $P_1$ and $P_2$ have distance at least $d$ in $G$. The other properties ensure that we can update the paths in the induction step without losing property \ref{itm:6}.

  Now we describe how to extend $P_1$ and $P_2$.  Suppose $P_1$ ends in $\alpha_i$, and $P_2$ ends in $r_i$ (the construction is analogous in the other case). Further, let $\beta_j$ be the other end of the component of $P_1 \cap P$ which contains $\alpha_i$. If $i = t$, then $P_1$ and $P_2$ are our desired $X - Y$ walks, since~$r_t, \alpha_t \in Y$.

    In order to ensure that $\dist_P(P_1 \cap P, P_2 \cap P) \ge 10d$ we need to distinguish several cases. In each case we update
    \[
    P_1:=P_1\alpha_iP\alpha_{k}A_{k}\ell_{k+1}Q_{k+1}r_{k+1} \text{ and } P_2:=P_2r_iB_i\beta_iP\alpha_{k+1}
    \]
    for some carefully chosen $k \in \{i, i+1, i+2, i+3\}$ where $\alpha_{k}$ occurs between $\alpha_i$ and $\beta_i$ along $P$ (see \cref{fig:extension}). 

    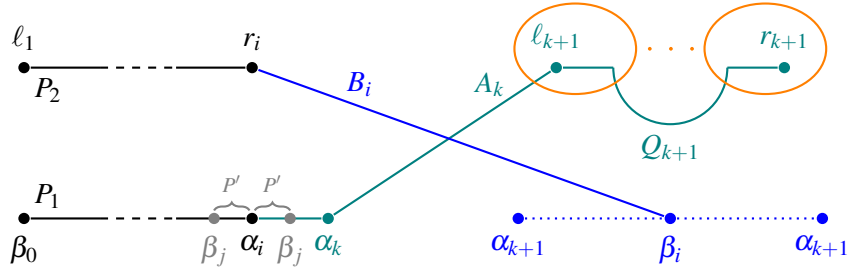
\begin{figure}[ht]
        \centering
        \begin{tikzpicture}
    
            \node[circle, fill=black, inner sep=1.5pt](l1) at (-5,2) {};
            \node[black, above = 0pt of l1]{$\ell_1$};
    
            \node[circle, fill=black, inner sep=1.5pt](ri) at (-2,2) {};
            \node[black, above = 0pt of ri]{$r_i$};
    
            \draw[black] (l1) -- (-4,2) node[near start, below] {$P_2$};
            \draw[black, dashed] (-4,2) -- (-3,2);
            \draw[black] (-3,2) -- (ri);
            
            \node[circle, fill=black, inner sep=1.5pt](b0) at (-5,0) {};
            \node[black, below = 0pt of b0]{$\beta_0$};
    
            \node[circle, fill=black, inner sep=1.5pt](ai) at (-2,0) {};
            \node[black, below = 0pt of ai]{$\alpha_i$};
    
            \draw[black] (b0) -- (-4, 0) node[near start, above] {$P_1$};
            \draw[black, dashed] (-4, 0) -- (-3, 0);
            \draw[black] (-3,0) -- (ai);

            \node[circle, fill=teal, inner sep=1.5pt](ak) at (-1,0) {};
            \node[teal, below = 0pt of ak]{$\alpha_k$};
    
            \node[circle, fill=teal, inner sep=1.5pt](lk+1) at (2,2) {};
            \node[teal, above = 0pt of lk+1]{$\ell_{k+1}$};
    
            \node[circle, fill=teal, inner sep=1.5pt](rk+1) at (5,2) {};
            \node[teal, above = 0pt of rk+1]{$r_{k+1}$};
    
            \draw[teal] (ai) -- (ak) -- (lk+1) node[near end, above] {$A_k$\;\,};
            \draw[teal] (lk+1) -- (2.75,2);
            \draw[teal] (2.75,2) arc (0:180:-0.75) node[midway, below] {$Q_{k+1}$};
            \draw[teal] (4.25,2) -- (rk+1);

            \draw[orange] (2.25,2.25) ellipse (0.8 and 0.6);
            \draw[orange] (4.75,2.25) ellipse (0.8 and 0.6);
            \node[circle, fill=orange, inner sep=0.5pt](dot) at (3.23,2.25) {};
            \node[circle, fill=orange, inner sep=0.5pt, right= 0.2 of dot](dot2) {};
            \node[circle, fill=orange, inner sep=0.5pt, right= 0.2 of dot2] {};

            \node[circle, fill=blue, inner sep=1.5pt](ak+1) at (1.5,0) {};
            \node[blue, below = 0pt of ak+1]{$\alpha_{k+1}$};
    
            \node[circle, fill=blue, inner sep=1.5pt](bi) at (3.5,0) {};
            \node[blue, below = 0pt of bi]{$\beta_i$};
    
            \node[circle, fill=blue, inner sep=1.5pt](ak+1') at (5.5,0) {};
            \node[blue, below = 0pt of ak+1']{$\alpha_{k+1}$};
    
            \draw[blue, dotted] (ak+1) -- (bi) -- (ak+1');
            \draw[blue] (ri)--(bi) node[near start, above] {$B_i$};

            \node[circle, fill=gray, inner sep=1.5pt](bj) at (-2.5,0) {};
            \node[gray, below = 0pt of bj]{$\beta_j$};
            \node[circle, fill=gray, inner sep=1.5pt](bj') at (-1.5,0) {};
            \node[gray, below = 0pt of bj']{$\beta_j$};
    
            \draw[gray,decorate,decoration={brace,amplitude=3pt}] ([yshift=3pt] bj.center) -- ([yshift=3pt,xshift=-1pt] ai.center) node[midway, above,yshift=3pt]{\scriptsize $P'$};
            \draw[gray,decorate,decoration={brace,amplitude=3pt}] ([yshift=3pt,xshift=1pt] ai.center) -- ([yshift=3pt] bj'.center) node[midway, above,yshift=3pt]{\scriptsize $P'$};
    
        \end{tikzpicture}
        
        \caption{The old paths $P_1$ and $P_2$ in black together with their updates in teal and blue, respectively. Note that $\alpha_{k+1}$ may occur before or after $\beta_i$ and that $\beta_j$ may occur before or after $\alpha_i$.}
        \label{fig:extension}
    \end{figure}
    
    Then properties \ref{itm:1} and \ref{itm:2} are clear from the definition of $P_1$ and $P_2$. By \ref{itm:3}, the component of $(P_1\alpha_i \cup P_2r_i) \cap P$ containing $\alpha_i$ occurs last on $P$. Since $\alpha_{k+1}$ occurs after $\alpha_i$ and $\alpha_k$ along $P$ and $\alpha_iP_1r_{k+1} \cap P = \alpha_iP\alpha_k$, this implies that \ref{itm:3} still holds for $P_1$ and $P_2$. To see that \ref{itm:4} holds, we first note that once we have shown \ref{itm:6}, it follows that the component $P'$ of $(P_1 \cup P_2) \cap P$ containing $\alpha_{k+1}$ is a subset of $P_2$. Since $\beta_i$ occurs after $\alpha_i$ along $P$ and $\alpha_i \in P_1$, property \ref{itm:3} of $P_1\alpha_i$ and $P_2r_i$ implies that $P'$ has ends $\beta_i$ and $\alpha_{k+1}$.
    We are thus left to check \ref{itm:5} and \ref{itm:6} in the individual cases. For this, we remark that since $\dist_P(\beta_i, \{\alpha_i, \beta_j\}) \geq 30d - 20d = 10d$ by assumption and \cref{dist13}, it suffices to ensure that $\dist_P(\beta_i, \alpha_{k}) \ge 10d$ if $k \neq i$ and that $\dist_P(\alpha_{k+1}, \{\alpha_{k}, \beta_j\}) \geq 10d$ in order to obtain \ref{itm:6}.

    To find the suitable $k \in \{ i,i+1,i+2,i+3\}$ for the update of $P_1$ and $P_2$, we consider where $\alpha_i, \alpha_{i+1}, \alpha_{i+2}, \alpha_{i+3}$ occur relative to $\beta_i$ on $P$.
    We depict the cases in \cref{fig:case1,fig:case2,fig:case3,fig:case4}.
    In each of the figures, we indicate the corresponding assumptions in orange and their implications in green.

  Suppose $\alpha_{i+1}$ occurs after $\beta_i$ along $P$. Then $\dist_P(\alpha_{i+1}, \{\alpha_i, \beta_j\}) \geq \dist_P(\beta_i, \{\alpha_i, \beta_j\}) \geq 10d$, and we set $k := i$ (see \cref{fig:case1}). By assumption, \ref{itm:5} and \ref{itm:6} are clear.

  \begin{figure}[ht]
    \centering
    \begin{tikzpicture}

        \def\distanceB[#1](#2:#3)(#4)(#5) 
            {
                \node[above=#4 of #2,xshift=0.3pt](2) {};
                \node[above=#4 of #3,xshift=-0.3pt](3) {};
                \draw[#1] (2.center) -- (3.center) node[midway,fill=white]{#5};
                \draw[#1] (2.center) -- ++(0,-0.15) -- ++(0,0.3);
                \draw[#1] (2.center) -- ++(0,0.15);
                \draw[#1] (3.center) -- ++(0,-0.15);
                \draw[#1] (3.center) -- ++(0,0.15);
            }

        \node[](left) at (-2, 0) {};
        \node[circle, fill=black, inner sep=1.5pt](ai) at (-1,0) {};
        \node[teal, below = 0pt of ai]{$\alpha_i$};
        \node[](topleft) at (-0.25, 1) {};

        \draw[teal] (left) -- (ai) node[midway,above,teal] {$P_1$} -- (topleft);

        \node[circle, fill=black, inner sep=1.5pt](ai+1) at (3, 0) {};
        \node[orange, below = 0pt of ai+1]{$\alpha_{i+1}$};
        \node[circle, fill=black, inner sep=1.5pt](bi) at (1,0) {};
        \node[blue, below = 0pt of bi]{$\beta_i$};
        \node[](topright) at (0.25, 1) {};

        \draw[blue] (topright) -- (bi) -- (ai+1) node[midway,above,blue] {$P_2$};

        \distanceB[green](ai:bi)(-1)($\ge 10d$)
    \end{tikzpicture}
    
    \caption{$\alpha_{i+1}$ occurs after $\beta_i$ along $P$.}\label{fig:case1}
\end{figure}
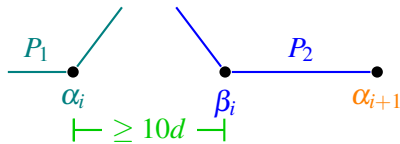

  Now suppose that $\alpha_{i+1}$ occurs before $\beta_i$ along $P$.
  Further, suppose $\alpha_{i+2}$ occurs after $\beta_i$ along $P$ (see \cref{fig:case2}). In this case, we proceed as follows. If $\dist_P(\alpha_{i+1}, \beta_i) \ge 10d$, then we set $k := i+1$ (see \cref{fig:case2a}). Again, \ref{itm:5} and \ref{itm:6} are clear from the assumption.
  We remark that by assumption this in particular is the case if $\dist_P(\alpha_i, \beta_j) > 10d$ and $\beta_j$ occurs after $\alpha_i$ along $P$, since then $\alpha_{i+1}$ occurs between $\alpha_i$ and $\beta_j$ along $P$ and, as shown above, $\dist_P(\beta_i,  \beta_j) \geq 10d$.

  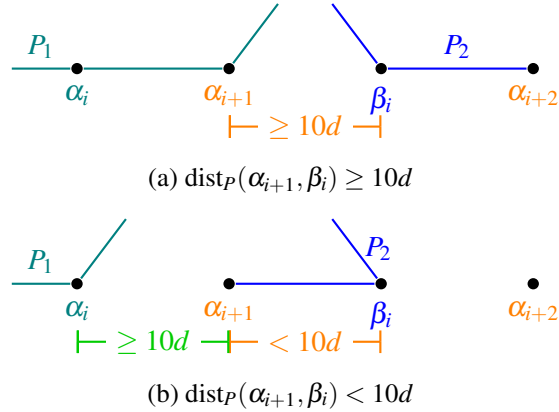
\begin{figure}[htb]
    \centering
    \begin{subfigure}[t]{\textwidth}
        \centering
        \begin{tikzpicture}
    
            \def\distanceB[#1](#2:#3)(#4)(#5) 
                {
                    \node[above=#4 of #2,xshift=0.3pt](2) {};
                    \node[above=#4 of #3,xshift=-0.3pt](3) {};
                    \draw[#1] (2.center) -- (3.center) node[midway,fill=white]{#5};
                    \draw[#1] (2.center) -- ++(0,-0.15) -- ++(0,0.3);
                    \draw[#1] (2.center) -- ++(0,0.15);
                    \draw[#1] (3.center) -- ++(0,-0.15);
                    \draw[#1] (3.center) -- ++(0,0.15);
                }
    
            \node[](left) at (-3, 0) {};
            \node[circle, fill=black, inner sep=1.5pt](ai) at (-2,0) {};
            \node[teal, below = 0pt of ai]{$\alpha_i$};
            \node[circle, fill=black, inner sep=1.5pt](ai+1) at (0,0) {};
            \node[orange, below = 0pt of ai+1]{$\alpha_{i+1}$};
            \node[](topleft) at (0.75, 1) {};
    
            \node[circle, fill=black, inner sep=1.5pt](ai+2) at (4, 0) {};
            \node[orange, below = 0pt of ai+2]{$\alpha_{i+2}$};
            \node[circle, fill=black, inner sep=1.5pt](bi) at (2,0) {};
            \node[blue, below = 0pt of bi]{$\beta_i$};
            \node[](topright) at (1.25, 1) {};
            
            \draw[teal] (left) -- (ai) node[midway,above,teal] {$P_1$} -- (ai+1) -- (topleft);
            \draw[blue] (topright) -- (bi) -- (ai+2) node[midway,above,blue] {$P_2$};
    
            \distanceB[orange](ai+1:bi)(-1)($\ge 10d$)
        \end{tikzpicture}
        \caption{$\dist_P(\alpha_{i+1}, \beta_i) \ge 10d$}
        \label{fig:case2a}
    \end{subfigure}
    \vfill
    \begin{subfigure}[b]{\textwidth}
        \centering
        \begin{tikzpicture}
    
            \def\distanceB[#1](#2:#3)(#4)(#5) 
                {
                    \node[above=#4 of #2,xshift=0.3pt](2) {};
                    \node[above=#4 of #3,xshift=-0.3pt](3) {};
                    \draw[#1] (2.center) -- (3.center) node[midway,fill=white]{#5};
                    \draw[#1] (2.center) -- ++(0,-0.15) -- ++(0,0.3);
                    \draw[#1] (2.center) -- ++(0,0.15);
                    \draw[#1] (3.center) -- ++(0,-0.15);
                    \draw[#1] (3.center) -- ++(0,0.15);
                }
    
            \node[](left) at (-3, 0) {};
            \node[circle, fill=black, inner sep=1.5pt](ai) at (-2,0) {};
            \node[teal, below = 0pt of ai]{$\alpha_i$};
            \node[circle, fill=black, inner sep=1.5pt](ai+1) at (0,0) {};
            \node[orange, below = 0pt of ai+1]{$\alpha_{i+1}$};
            \node[](topleft) at (-1.25, 1) {};
    
            \node[circle, fill=black, inner sep=1.5pt](ai+2) at (4, 0) {};
            \node[orange, below = 0pt of ai+2]{$\alpha_{i+2}$};
            \node[circle, fill=black, inner sep=1.5pt](bi) at (2,0) {};
            \node[blue, below = 0pt of bi]{$\beta_i$};
            \node[](topright) at (1.25, 1) {};

            \draw[teal] (left) -- (ai) node[midway,above,teal] {$P_1$} -- (topleft);
            \draw[blue] (ai+1) -- (bi) -- (topright) node[midway,right,blue] {$P_2$};

            \distanceB[orange](ai+1:bi)(-1)($< 10d$)
            \distanceB[green](ai:ai+1)(-1)($\ge 10d$)
        \end{tikzpicture}
         \caption{$\dist_P(\alpha_{i+1}, \beta_i) < 10d$}
         \label{fig:case2b}
    \end{subfigure}
    
    \caption{$\alpha_{i+1}$ occurs before $\beta_i$ along $P$ and $\alpha_{i+2}$ occurs after $\beta_i$ along $P$.}
    \label{fig:case2}
\end{figure}

  Otherwise, $\dist_P(\alpha_{i+1}, \{\alpha_i, \beta_j\}) \ge 10d$. Indeed, either $\beta_j$ occurs before $\alpha_i$ along $P$ and thus we have $\dist_P(\alpha_{i+1}, \{\alpha_i, \beta_j\}) = \dist_P(\alpha_{i+1}, \alpha_i) = \dist_P(\alpha_i, \beta_i) - \dist_P(\alpha_{i+1}, \beta_i) \ge 30d - 10d \ge 10d$ by \cref{dist13}, or $\beta_j$ occurs after $\alpha_i$ along $P$ and thus we have  $\dist_P(\alpha_{i+1}, \{\alpha_i, \beta_j\}) = \dist_P(\alpha_{i+1}, \beta_j) = \dist_P(\alpha_i, \beta_i) - \dist_P(\alpha_{i+1}, \beta_i) - \dist_P(\alpha_i, \beta_j) \ge 30d - 10d -10d \ge 10d$.
  In this case we set $k := i$ (see \cref{fig:case2b}). Then \ref{itm:6} holds by the previous calculation. Moreover, recall that $\dist_P(\alpha_{k+1}, \beta_i) < 10d$ by assumption, and thus also \ref{itm:5} holds.

  We may thus assume that both $\alpha_{i+1}$ and $\alpha_{i+2}$ occur before $\beta_i$ along $P$. Suppose further that $\alpha_{i+3}$ occurs after $\beta_i$ along $P$ (see \cref{fig:case3}). In this case, we proceed as follows. If $\dist_P(\alpha_{i+2}, \beta_i) \ge 10d$, then we set $k := i+2$ (see \cref{fig:case3a}). Then \ref{itm:5} and \ref{itm:6} are clear from the assumption.
  
\begin{figure}[ht]
    \centering
    \begin{subfigure}[t]{\textwidth}
        \centering
        \begin{tikzpicture}
    
            \def\distanceB[#1](#2:#3)(#4)(#5) 
                {
                    \node[above=#4 of #2,xshift=0.3pt](2) {};
                    \node[above=#4 of #3,xshift=-0.3pt](3) {};
                    \draw[#1] (2.center) -- (3.center) node[midway,fill=white]{#5};
                    \draw[#1] (2.center) -- ++(0,-0.15) -- ++(0,0.3);
                    \draw[#1] (2.center) -- ++(0,0.15);
                    \draw[#1] (3.center) -- ++(0,-0.15);
                    \draw[#1] (3.center) -- ++(0,0.15);
                }
    
            \node[circle, fill=black, inner sep=1.5pt](ai) at (-2,0) {};
            \node[black, below = 0pt of ai]{$\alpha_i$};
            \node[circle, fill=black, inner sep=1.5pt](ai+1) at (0,0) {};
            \node[orange, below = 0pt of ai+1]{$\alpha_{i+1}$};
            \node[circle, fill=black, inner sep=1.5pt](ai+2) at (2,0) {};
            \node[orange, below = 0pt of ai+2]{$\alpha_{i+2}$};
            \node[circle, fill=black, inner sep=1.5pt](bi) at (4,0) {};
            \node[black, below = 0pt of bi]{$\beta_i$};
            \node[circle, fill=black, inner sep=1.5pt](ai+3) at (6, 0) {};
            \node[orange, below = 0pt of ai+3]{$\alpha_{i+3}$};
    
            \node[](left) at (-3, 0) {};
            \node[](right) at (7, 0) {};
    
            \node[](topleft) at (2.75, 1) {};
            \node[](topright) at (3.25, 1) {};
    
            \draw[teal] (ai) -- (ai+1) -- (ai+2) node[midway,above,teal] {$P_1$} -- (topleft);
            \draw[blue] (topright) -- (bi) -- (ai+3) node[midway,above,blue] {$P_2$};
    
            \distanceB[orange](ai+2:bi)(-1)($\ge 10d$)
        \end{tikzpicture}
        \caption{$\dist_P(\alpha_{i+2}, \beta_i) \ge 10d$}
        \label{fig:case3a}
    \end{subfigure}
    \vfill
    \begin{subfigure}[b]{\textwidth}
        \centering
        \begin{tikzpicture}
    
            \def\distanceB[#1](#2:#3)(#4)(#5) 
                {
                    \node[above=#4 of #2,xshift=0.3pt](2) {};
                    \node[above=#4 of #3,xshift=-0.3pt](3) {};
                    \draw[#1] (2.center) -- (3.center) node[midway,fill=white]{#5};
                    \draw[#1] (2.center) -- ++(0,-0.15) -- ++(0,0.3);
                    \draw[#1] (2.center) -- ++(0,0.15);
                    \draw[#1] (3.center) -- ++(0,-0.15);
                    \draw[#1] (3.center) -- ++(0,0.15);
                }
    
            \node[circle, fill=black, inner sep=1.5pt](ai) at (-2,0) {};
            \node[black, below = 0pt of ai]{$\alpha_i$};
            \node[circle, fill=black, inner sep=1.5pt](ai+1) at (0,0) {};
            \node[orange, below = 0pt of ai+1]{$\alpha_{i+1}$};
            \node[circle, fill=black, inner sep=1.5pt](ai+2) at (2,0) {};
            \node[orange, below = 0pt of ai+2]{$\alpha_{i+2}$};
            \node[circle, fill=black, inner sep=1.5pt](bi) at (4,0) {};
            \node[black, below = 0pt of bi]{$\beta_i$};
            \node[circle, fill=black, inner sep=1.5pt](ai+3) at (6, 0) {};
            \node[orange, below = 0pt of ai+3]{$\alpha_{i+3}$};
    
            \node[](left) at (-3, 0) {};
            \node[](right) at (7, 0) {};
    
            \node[](topleft) at (0.75, 1) {};
            \node[](topright) at (3.25, 1) {};
    
            \draw[teal] (ai) -- (ai+1) node[midway,above,teal] {$P_1$} -- (topleft);
            \draw[blue] (topright) -- (bi) node[midway,right,blue] {$P_2$} -- (ai+2);
    
            \distanceB[orange](ai+2:bi)(-1)($< 10d$)
            \distanceB[orange](ai+1:ai+2)(-1.5)($\ge 10d$)
            \distanceB[green](ai:ai+2)(-1)($\ge 10d$)
            \distanceB[gray](ai:ai+2)(-2)($\ge d_2$)
        \end{tikzpicture}
        \caption{$\dist_P(\alpha_{i+2}, \beta_i) < 10d$ and $\dist_P(\alpha_{i+1}, \alpha_{i+2}) \ge 10d$}
        \label{fig:case3b}
    \end{subfigure}
    \vfill
    \begin{subfigure}[b]{\textwidth}
        \centering
        \begin{tikzpicture}
    
            \def\distanceB[#1](#2:#3)(#4)(#5) 
                {
                    \node[above=#4 of #2,xshift=0.3pt](2) {};
                    \node[above=#4 of #3,xshift=-0.3pt](3) {};
                    \draw[#1] (2.center) -- (3.center) node[midway,fill=white]{#5};
                    \draw[#1] (2.center) -- ++(0,-0.15) -- ++(0,0.3);
                    \draw[#1] (2.center) -- ++(0,0.15);
                    \draw[#1] (3.center) -- ++(0,-0.15);
                    \draw[#1] (3.center) -- ++(0,0.15);
                }
    
            \node[circle, fill=black, inner sep=1.5pt](ai) at (-2,0) {};
            \node[black, below = 0pt of ai]{$\alpha_i$};
            \node[circle, fill=black, inner sep=1.5pt](ai+1) at (0,0) {};
            \node[orange, below = 0pt of ai+1]{$\alpha_{i+1}$};
            \node[circle, fill=black, inner sep=1.5pt](ai+2) at (2,0) {};
            \node[orange, below = 0pt of ai+2]{$\alpha_{i+2}$};
            \node[circle, fill=black, inner sep=1.5pt](bi) at (4,0) {};
            \node[black, below = 0pt of bi]{$\beta_i$};
            \node[circle, fill=black, inner sep=1.5pt](ai+3) at (6, 0) {};
            \node[orange, below = 0pt of ai+3]{$\alpha_{i+3}$};
    
            \node[](left) at (-3, 0) {};
            \node[](right) at (7, 0) {};
    
            \node[](topleft) at (-1.25, 1) {};
            \node[](topright) at (3.25, 1) {};
    
            \draw[teal] (ai) -- (topleft) node[midway,left,teal] {$P_1$};
            \draw[blue] (ai+1) -- (ai+2) -- (bi) -- (topright) node[midway,right,blue] {$P_2$};
    
            \distanceB[green](ai:ai+1)(-1)($\ge 10d$)
            \distanceB[orange](ai+1:ai+2)(-1)($< 10d$)
            \distanceB[orange](ai+2:bi)(-1)($< 10d$)
            \distanceB[gray](ai:ai+2)(-1.5)($\ge d_2$)
        \end{tikzpicture}
        \caption{$\dist_P(\alpha_{i+2}, \beta_i) < 10d$ and $\dist_P(\alpha_{i+1}, \alpha_{i+2}) < 10d$}
        \label{fig:case3c}
    \end{subfigure}
    \caption{$\alpha_{i+1}$ and $\alpha_{i+2}$ occur before $\beta_i$ along $P$ and $\alpha_{i+3}$ occurs after $\beta_i$ along $P$.}
    \label{fig:case3}
\end{figure}
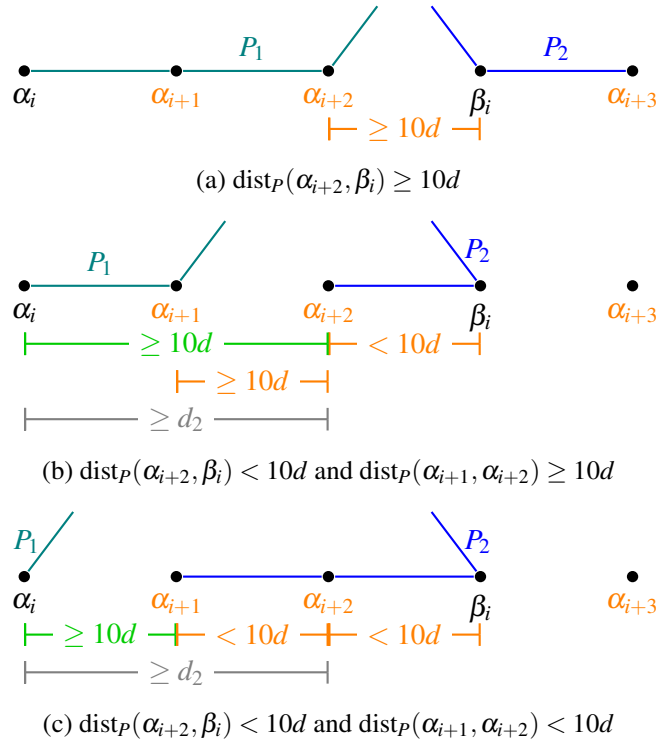

  Otherwise, if $\dist_P(\alpha_{i+2}, \beta_i) < 10d$, we have $\dist_P(\alpha_{i+2}, \{\alpha_i, \beta_j\}) \ge \dist_P(\alpha_{i+2}, \alpha_i) - 20d \ge d_2 - 20d \ge 10d$ by assumption and Claim \ref{dist42}. If $\dist_P(\alpha_{i+1}, \alpha_{i+2}) \ge 10d$, then we set $k := i+1$ (see \cref{fig:case3b}). Then \ref{itm:6} holds by assumption. Since also $\dist_P(\alpha_{k+1}, \beta_i) < 10d$ by assumption, \ref{itm:5} holds too.

  Otherwise, $\dist_P(\alpha_{i+1}, \beta_i) = \dist_P(\alpha_{i+1}, \alpha_{i+2}) + \dist_P(\alpha_{i+2}, \beta_{i}) \le 20d$ and $\dist_P(\alpha_{i+1}, \{\alpha_i, \beta_j\}) \ge \dist_P(\alpha_i, \alpha_{i+2}) - \dist_P(\alpha_{i+1}, \alpha_{i+2}) - 20d \ge d_2 - 10d - 20d \ge 10d$ by assumption and Claim \ref{dist42}, and we then set $k := i$ (see \cref{fig:case3c}). Then \ref{itm:6} holds by the previous calculation. Moreover, \ref{itm:5} holds since $\alpha_{k+2} = v_{a(H'_{w(k+2)+1})}$ occurs between $\alpha_{k+1} = v_{a(H'_{w(k+1)+1})}$ and $\beta_i$ by assumption and because $a(H'_{w(k+1)+1}) < a(H'_{w(k+2)+1})$.

  Thus, we may assume that $\alpha_{i+1}$, $\alpha_{i+2}$ and $\alpha_{i+3}$ occur before $\beta_i$ along $P$. We claim that $\alpha_{i+4}$ occurs after~$\beta_i$ along $P$ and that $\dist_P(\alpha_{i+1}, \alpha_{i+2}) \geq d_2$.
  Indeed, the fact that $\alpha_{i+3} = v_{a(H'_{w(i+3)+1})}$  occurs before $\beta_i=v_{b(H_{w(i)}')}$ implies that no three of $H'_{w(i)}, \dots, H'_{w(i+3)+1}$ are in $\mathcal{M}$ since $\mathcal{M}$ is clean.
  Note that $w(i+3)+1 \geq w(i) + 4$, since the $w(i)$ are strictly increasing by definition.
  By construction of $(H'_i)_{i \in [m+s]}$, we have that if $H'_i \in \mathcal{S}$ then $H'_{i+1} \in \mathcal{M}$.
  Hence, $w(i+3)+1 = w(i) + 4$, and $H'_{w(i)}, H'_{w(i)+2}, H'_{w(i)+4} \in \mathcal{S}$, and $H'_{w(i)+1}, H'_{w(i)+3} \in \mathcal{M}$. In particular, $\mathcal{W}^{i+1} = H_{w(i)+1}', \mathcal{W}^{i+2} = H_{w(i)+2}'$ and $\mathcal{W}^{i+3} = H_{w(i)+3}'$ consist of single components of $H'$. It follows that $\dist_P(\alpha_{i+1}, \alpha_{i+2}) = \dist_P(v_{a(H'_{w(i)+2})}, v_{a(H'_{w(i)+3})}) \ge d_2$ by the choice of $\mathcal{S}$.
  Moreover, as $H'_{w(i)+4} \in \mathcal{S}$, we have $H'_{w(i)+5} \in \mathcal{M}$. Together with $H'_{w(i)+1}, H'_{w(i)+3} \in \mathcal{M}$, this implies that $a(H'_{w(i+4)+1}) \geq a(H'_{w(i)+5}) \geq b(H'_{w(i)+1}) \geq b(H'_{w(i)})$ as $\mathcal{M}$ is clean.
  Thus, $\alpha_{i+4} = v_{a(H_{w(i+4)+1})}$ occurs after $\beta_i = v_{b(H'_{w(i)})}$ along $P$.

  \begin{figure}[ht]
    \centering
    \begin{subfigure}[t]{\textwidth}
        \centering
        \begin{tikzpicture}
    
            \def\distanceB[#1](#2:#3)(#4)(#5) 
                {
                    \node[above=#4 of #2,xshift=0.3pt](2) {};
                    \node[above=#4 of #3,xshift=-0.3pt](3) {};
                    \draw[#1] (2.center) -- (3.center) node[midway,fill=white]{#5};
                    \draw[#1] (2.center) -- ++(0,-0.15) -- ++(0,0.3);
                    \draw[#1] (2.center) -- ++(0,0.15);
                    \draw[#1] (3.center) -- ++(0,-0.15);
                    \draw[#1] (3.center) -- ++(0,0.15);
                }
    
            \node[circle, fill=black, inner sep=1.5pt](ai) at (-2,0) {};
            \node[black, below = 0pt of ai]{$\alpha_i$};
            \node[circle, fill=black, inner sep=1.5pt](ai+1) at (0,0) {};
            \node[orange, below = 0pt of ai+1]{$\alpha_{i+1}$};
            \node[circle, fill=black, inner sep=1.5pt](ai+2) at (2,0) {};
            \node[orange, below = 0pt of ai+2]{$\alpha_{i+2}$};
            \node[circle, fill=black, inner sep=1.5pt](ai+3) at (4, 0) {};
            \node[orange, below = 0pt of ai+3]{$\alpha_{i+3}$};
            \node[circle, fill=black, inner sep=1.5pt](bi) at (6, 0) {};
            \node[black, below = 0pt of bi]{$\beta_i$};
            \node[circle, fill=black, inner sep=1.5pt](ai+4) at (8, 0) {};
            \node[green, below = 0pt of ai+4]{$\alpha_{i+4}$};
    
            \node[](left) at (-3, 0) {};
            \node[](right) at (7, 0) {};
    
            \node[](topleft) at (0.75, 1) {};
            \node[](topright) at (5.25, 1) {};
    
            \draw[teal] (ai) -- (ai+1) node[midway,above,teal] {$P_1$} -- (topleft);
            \draw[blue] (topright) -- (bi) node[midway,right,blue] {$P_2$} -- (ai+3) -- (ai+2);
    
            \distanceB[green](ai+1:ai+2)(-1)($\ge d_2$)
            \distanceB[orange](ai+2:bi)(-1)($\le 20d$)
        \end{tikzpicture}
        \caption{$\dist_P(\alpha_{i+2}, \beta_i) \le 20d$}
        \label{fig:case4a}
    \end{subfigure}
    \vfill
    \begin{subfigure}[b]{\textwidth}
        \centering
        \begin{tikzpicture}
    
            \def\distanceB[#1](#2:#3)(#4)(#5) 
                {
                    \node[above=#4 of #2,xshift=0.3pt](2) {};
                    \node[above=#4 of #3,xshift=-0.3pt](3) {};
                    \draw[#1] (2.center) -- (3.center) node[midway,fill=white]{#5};
                    \draw[#1] (2.center) -- ++(0,-0.15) -- ++(0,0.3);
                    \draw[#1] (2.center) -- ++(0,0.15);
                    \draw[#1] (3.center) -- ++(0,-0.15);
                    \draw[#1] (3.center) -- ++(0,0.15);
                }
    
            \node[circle, fill=black, inner sep=1.5pt](ai) at (-2,0) {};
            \node[black, below = 0pt of ai]{$\alpha_i$};
            \node[circle, fill=black, inner sep=1.5pt](ai+1) at (0,0) {};
            \node[orange, below = 0pt of ai+1]{$\alpha_{i+1}$};
            \node[circle, fill=black, inner sep=1.5pt](ai+2) at (2,0) {};
            \node[orange, below = 0pt of ai+2]{$\alpha_{i+2}$};
            \node[circle, fill=black, inner sep=1.5pt](ai+3) at (4, 0) {};
            \node[orange, below = 0pt of ai+3]{$\alpha_{i+3}$};
            \node[circle, fill=black, inner sep=1.5pt](bi) at (6, 0) {};
            \node[black, below = 0pt of bi]{$\beta_i$};
            \node[circle, fill=black, inner sep=1.5pt](ai+4) at (8, 0) {};
            \node[green, below = 0pt of ai+4]{$\alpha_{i+4}$};
    
            \node[](left) at (-3, 0) {};
            \node[](right) at (7, 0) {};
    
            \node[](topleft) at (4.75, 1) {};
            \node[](topright) at (5.25, 1) {};
    
            \draw[teal] (ai) -- (ai+1) node[midway,above,teal] {$P_1$} -- (ai+2) -- (ai+3) -- (topleft);
            \draw[blue] (topright) -- (bi) node[midway,right,blue] {$P_2$} -- (ai+4);
    
            \distanceB[orange](ai+3:bi)(-1)($\ge 10d$)
        \end{tikzpicture}
        \caption{$\dist_P(\alpha_{i+3}, \beta_i) \ge 10d$}
        \label{fig:case4b}
    \end{subfigure}
    \vfill
    \begin{subfigure}[b]{\textwidth}
        \centering
        \begin{tikzpicture}
    
            \def\distanceB[#1](#2:#3)(#4)(#5) 
                {
                    \node[above=#4 of #2,xshift=0.3pt](2) {};
                    \node[above=#4 of #3,xshift=-0.3pt](3) {};
                    \draw[#1] (2.center) -- (3.center) node[midway,fill=white]{#5};
                    \draw[#1] (2.center) -- ++(0,-0.15) -- ++(0,0.3);
                    \draw[#1] (2.center) -- ++(0,0.15);
                    \draw[#1] (3.center) -- ++(0,-0.15);
                    \draw[#1] (3.center) -- ++(0,0.15);
                }
    
            \node[circle, fill=black, inner sep=1.5pt](ai) at (-2,0) {};
            \node[black, below = 0pt of ai]{$\alpha_i$};
            \node[circle, fill=black, inner sep=1.5pt](ai+1) at (0,0) {};
            \node[orange, below = 0pt of ai+1]{$\alpha_{i+1}$};
            \node[circle, fill=black, inner sep=1.5pt](ai+2) at (2,0) {};
            \node[orange, below = 0pt of ai+2]{$\alpha_{i+2}$};
            \node[circle, fill=black, inner sep=1.5pt](ai+3) at (4, 0) {};
            \node[orange, below = 0pt of ai+3]{$\alpha_{i+3}$};
            \node[circle, fill=black, inner sep=1.5pt](bi) at (6, 0) {};
            \node[black, below = 0pt of bi]{$\beta_i$};
            \node[circle, fill=black, inner sep=1.5pt](ai+4) at (8, 0) {};
            \node[green, below = 0pt of ai+4]{$\alpha_{i+4}$};
    
            \node[](left) at (-3, 0) {};
            \node[](right) at (7, 0) {};
    
            \node[](topleft) at (2.75, 1) {};
            \node[](topright) at (5.25, 1) {};
    
            \draw[teal] (ai) -- (ai+1) node[midway,above,teal] {$P_1$} -- (ai+2) -- (topleft);
            \draw[blue] (topright) -- (bi) node[midway,right,blue] {$P_2$} -- (ai+3);
    
            \distanceB[orange](ai+2:bi)(-1)($> 20d$)
            \distanceB[green](ai+2:ai+3)(-1.5)($\ge 10d$)
            \distanceB[orange](ai+3:bi)(-1.5)($< 10d$)
        \end{tikzpicture}
        \caption{$\dist_P(\alpha_{i+2}, \beta_i) > 20d$ and $\dist_P(\alpha_{i+3}, \beta_i) < 10d$}
        \label{fig:case4c}
    \end{subfigure}
    \caption{$\alpha_{i+1}$, $\alpha_{i+2}$ and $\alpha_{i+3}$ occur before $\beta_i$ along $P$.}
    \label{fig:case4}
\end{figure}
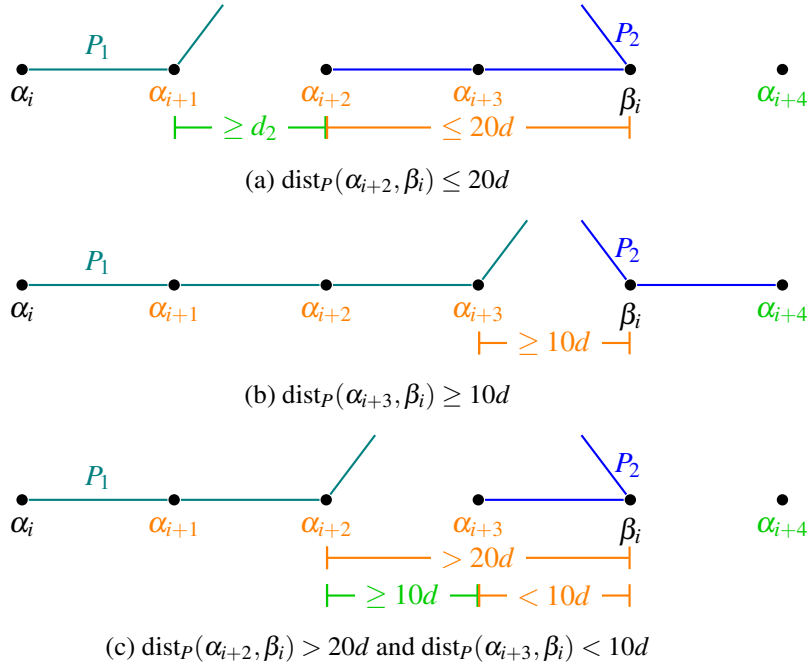
  
  If $\dist_P(\alpha_{i+2}, \beta_i) \le 20d$, then we set $k := i+1$ (see \cref{fig:case4a}). Since $\dist_P(\alpha_{i+1}, \alpha_{i+2}) \geq d_2 \geq 10d$ as shown above, \ref{itm:6} holds. Moreover, $\alpha_{k+2} = v_{a(H'_{w(k+2)+1})}$ occurs between $\alpha_{k+1} = v_{a(H'_{w(k+1)+1})}$ and $\beta_i$ by assumption and because $a(H'_{w(k+1)+1}) < a(H'_{w(k+2)+1})$. Thus \ref{itm:5} holds.

  If $\dist_P(\alpha_{i+3}, \beta_i) \ge 10d$, then we set $k := i+3$ (see \cref{fig:case4b}). Then \ref{itm:5} and \ref{itm:6} are clear by assumption.

  Otherwise, $\dist_P(\alpha_{i+2}, \alpha_{i+3}) = \dist_P(\alpha_{i+2}, \beta_i) - \dist_P(\alpha_{i+3}, \beta_i) \ge 20d - 10d = 10d$. In this case we set $k := i+2$ (see \cref{fig:case4c}). Then \ref{itm:6} holds by the previous calculation. Moreover, \ref{itm:5} holds since $\dist_P(\alpha_{i+3}, \beta_i) < 10d$ by assumption.

  \begin{claim} \label{10dbound}
  $\dist_G(P_1 \cap P, P_2 \cap P) \geq 10d$.
  \end{claim}

  \begin{proof}
   By construction, we have $\dist_P(P_1 \cap P, P_2 \cap P) \geq 10d$.  Since $P$ is a shortest $X-Y$ path, the claim immediately follows.
  \end{proof}

\begin{claim}
    $\dist_G(P_1, P_2) \ge d$.
\end{claim}

\begin{proof}
    Suppose for a contradiction that $p_1 \in V(P_1)$, $p_2 \in V(P_2)$, and $\dist_G(p_1, p_2) \le d-1$.
    If $p_1 \in V(Q_i)$ and $p_2 \in V(Q_j)$ for some $i \neq j$, then $\dist_G(p_1, p_2) \ge d$ by the minimality of $t$.  Suppose $p_1 \in V(P)$ and $p_2 \in V(Q_i)$ for some $i$.  By~\cref{closetoH'} and~\cref{sixpaths}, $\dist_G(p_2, H') \le 4d$. Since \mbox{$\dist_G(P, H') \ge d_1$}, we have $\dist_G(p_1, p_2) \ge d_1-4d \ge d$. Suppose $p_1 \in V(P)$ and $p_2 \in V(A_i)$ for some $i$.  By~\cref{10dbound}, $\dist_G(p_1, \alpha_i) \ge 10d$. This contradicts $\dist_G(p_1, \alpha_i) \le \dist_G(p_1, p_2)+\dist_G(p_2, \alpha_i) \le d-1+d_1 < 10d$.  The same argument handles the case $p_1 \in V(P)$ and $p_2 \in V(B_i)$ for some $i$.

    Suppose $p_1 \in V(A_i)$ and $p_2 \in V(A_j)$ for some $i<j$. If $\dist_G(p_1, \ell_i) + \dist_G(p_2, \ell_j) \le d+1$, then
    \[
    \dist_G(\ell_i, \ell_j) \le \dist_G(\ell_i, p_1) + \dist_G(p_1, p_2) + \dist_G(p_2, \ell_j) \le (d+1) + (d-1) = 2d,
    \]
    which contradicts the minimality of $t$. Otherwise,
    \begin{align*}
     \dist_G(\alpha_i, \alpha_j) &\le \dist_G(\alpha_i, p_1)+\dist_G(p_1, p_2) +\dist_G(p_2, \alpha_j)\\
     & = (\dist_G(\alpha_i, \ell_i) - \dist_G(\ell_i, p_1)) + \dist_G(p_1, p_2) + (\dist_G(\alpha_j, \ell_j) - \dist_G(\ell_j, p_2))\\
     & \le 2d_1 + 2 -(d+2) + (d-1) < 10d,
    \end{align*}
    which contradicts $\dist_G(\alpha_i, \alpha_j) \ge 10d$ by~\cref{10dbound}.
    The same argument handles the case $p_1 \in V(B_i)$ and $p_2 \in V(B_j),$ and $p_1 \in V(A_i)$ and $p_2 \in V(B_j)$, for some $i \neq j$.

    Suppose $p_1 \in V(Q_i)$ and $p_2 \in V(A_j)$ for some $i,j$. Since $\dist_G(H', P)=d_1+1$ and~$||A_j|| \le d_1 + 1$,
    \[
    \dist_G(P,p_2) + \dist_G(p_2,\ell_j) = ||A_j|| \leq \dist_G(H',P) \leq \dist_G(H',p_1) + \dist_G(p_1,p_2) + \dist_G(p_2,P).
    \]
    Moreover, assume that the component of $\mathcal{W}^i \setminus E(H')$ which contains~$p_1$ consists of exactly $c$ composite paths.
    By~\cref{closetoH'}, $\dist_{\mathcal{W}^i}(p_1, V(H')) \le cd$.
    Altogether, we have $\dist_{G}(p_2, \ell_j) \le \dist_G(H', p_1)+\dist_G(p_1, p_2) \le cd+d-1$.
    Therefore,
    \[
    \dist_{G}(p_1, \ell_j) \le \dist_G(p_1, p_2)+\dist_{G}(p_2, \ell_j) \le cd +2d-2 \le (c+2)d.
    \]
    However, this contradicts the minimality of $t$, since there is $\mathcal{W}^i-\mathcal{W}^j$ path of length at most $(c+2)d$.

     The same argument also handles the case $p_1 \in V(Q_i)$ and $p_2 \in V(B_j)$ for some $i,j$.  By symmetry between $p_1$ and $p_2$, there are no more remaining cases.
\end{proof}
This concludes the proof.
\end{proof}

\section*{Acknowledgements}

We are grateful to Rose McCarty and Paul Seymour for allowing us to include the proof of~\cref{d=3reduction} in this paper.  We  also thank Agelos Georgakopoulos and Panagiotis Papazoglou for informing us of the existence of their preprint~\cite{GP23}. We also acknowledge 
Adrian Thananopavarn for correcting the lowerbound in~\cref{lem:lowerbound} from an earlier version of this paper. Finally, we thank the two anonymous referees for their valuable input which greatly improved the readability of the paper.  

The third author gratefully acknowledges support by doctoral scholarships of the Studienstiftung des deutschen Volkes and the Cusanuswerk -- Bisch\"{o}fliche Studienf\"{o}rderung.
The fourth author gratefully acknowledges support by a doctoral scholarship of the Studienstiftung des deutschen Volkes.

\bibliographystyle{abbrv}
\bibliography{references}
\end{document}